\documentclass[a4paper,english,10pt]{scrartcl}
\pdfoutput=1

\KOMAoptions{DIV=9,headings=normal,abstract=true}

\usepackage[T1]{fontenc}
\usepackage[utf8]{inputenc}
\usepackage[english]{babel}
\usepackage{amsmath,amssymb,amsfonts,amsthm}

\usepackage{microtype}
\usepackage{mathtools}  %
\usepackage{booktabs}   %
\usepackage{array}
\newcolumntype{R}[1]{>{\raggedleft\let\newline\\\arraybackslash\hspace{0pt}}p{#1}}

\usepackage[]{csquotes}
\usepackage{xspace}

\usepackage{float}

\usepackage[]{graphicx}

\usepackage[]{url}
\usepackage[backgroundcolor=white,linecolor=red]{todonotes}

\usepackage[shortlabels,inline]{enumitem}
\setlist[description]{labelindent=!,labelsep=1em,font=\normalfont\bfseries}

\setkomafont{sectioning}{\bfseries\rmfamily}
\setkomafont{title}{}

\usepackage[
  citestyle=authortitle-icomp,
  bibstyle=verbose,
  backend=bibtex,
  isbn=false,
  opcittracker=true,
  autocite=footnote,
  doi=true,url=false,
  dashed=true,
  maxnames=5
]{biblatex}
\bibliography{articles}
\DeclareNameAlias{sortname}{first-last}
\usepackage[ruled]{caption}      %
\usepackage{setspace}
\setdisplayskipstretch{.6}

\newcommand\blfootnote[1]{%
  \begingroup
  \renewcommand\thefootnote{}\footnote{#1}%
  \addtocounter{footnote}{-1}%
  \endgroup
}

\usepackage[]{algpseudocode}
\floatstyle{ruled}
\newfloat{algo}{tp}{lop}
\floatname{algo}{Algorithm}
\allowdisplaybreaks[4]

\newtheorem{proposition}{Proposition}
\newtheorem{lemma}[proposition]{Lemma}
\newtheorem{theorem}[proposition]{Theorem}

\newtheorem*{question}{Problem}

 \def\mbbone{\mathbf{1}} %

\theoremstyle{definition}
\newtheorem{definition}[proposition]{Definition}

\theoremstyle{remark}
\newtheorem*{remark}{Remark}

\newcommand{\mop}[1]{\operatorname{#1}}
\newcommand{\ud}{\mathrm{d}}
\newcommand{\st}{\ \middle|\ }
\newcommand{\eqdef}{\smash{\ \stackrel{\text{def}}{=}\ }}
\newcommand{\abs}[1]{\left|#1\right|}

\newcommand{\bE}{{\mathbb{E}}}
\newcommand{\bP}{{\mathbb{P}}}
\newcommand{\bN}{{\mathbb{N}}}
\newcommand{\bC}{{\mathbb{C}}}

\newcommand{\bR}{{\mathbb{R}}}

\newcommand{\bS}{{\mathbb{S}}}

\newcommand{\cN}{{\mathcal{N}}}
\newcommand{\cO}{{\mathcal{O}}}
\newcommand{\cH}{{\mathcal{H}}}

\newcommand{\vol}{\mop{vol}}

\newcommand{\HCp}{\ensuremath{\mop{HC}'}}

\newcommand{\BP}{\ensuremath{\mop{BP}}}

\newcommand{\Sin}{\mop{Sin}}
\newcommand{\Cos}{\mop{Cos}}

\newcommand{\Proj}{\ensuremath{\bP^n}}

\newcommand{\polsys}{\cH}
\newcommand{\SH}{{\bS(\cH)}}

\newcommand{\rstd}{\rho_\textrm{std}}

\newcommand{\flo}[1]{\lfloor #1 \rfloor}

\def\geq{\geqslant}
\def\leq{\leqslant}

\def\Re{\mop{Re}}

\def\faCalculator{\raisebox{-.3ex}{\includegraphics[height=1em]{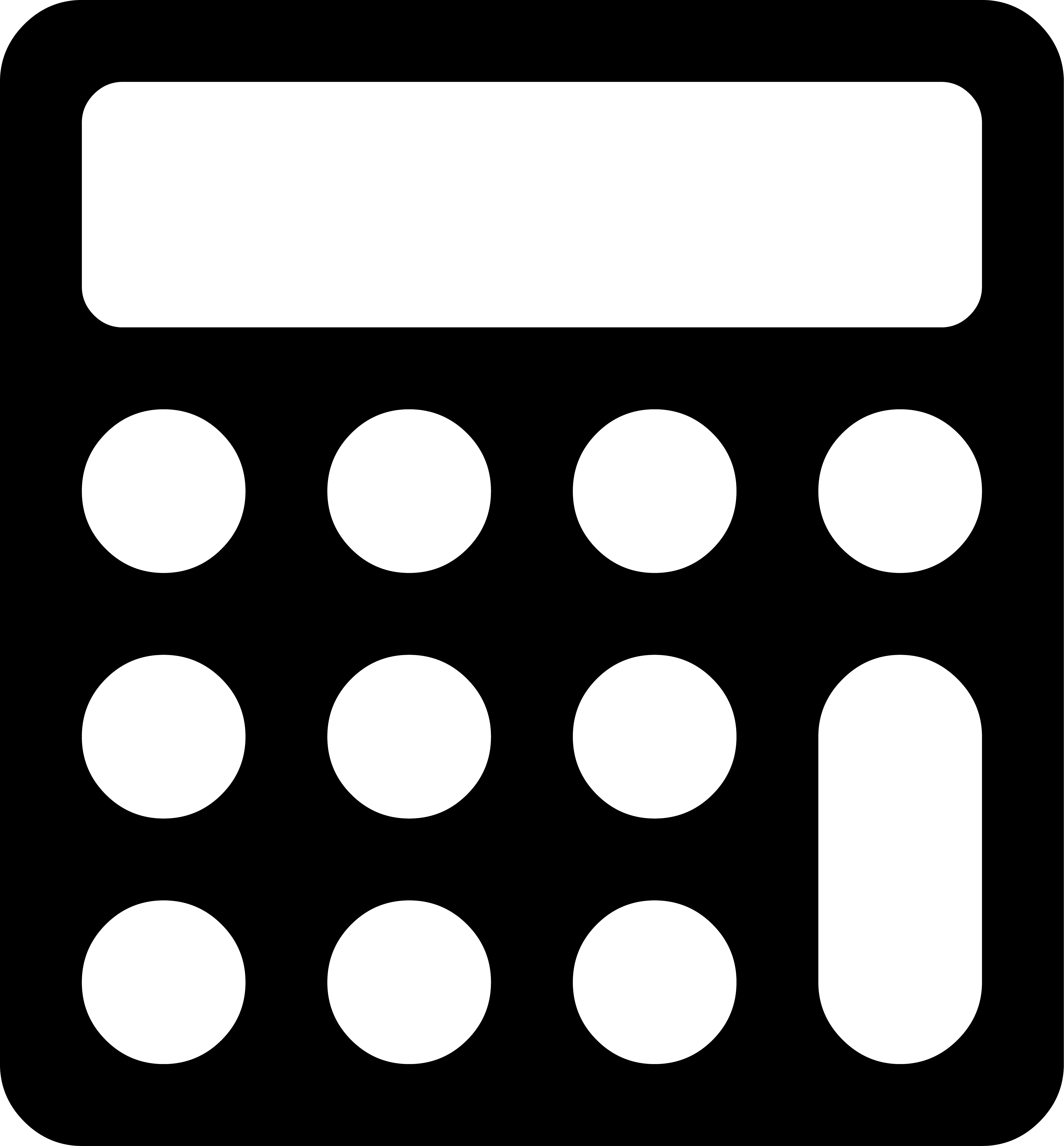}}}
\def\calcsub{{\scalebox{.5}{\text{\faCalculator}}}}

\def\leqcalc{ \leq_\calcsub }
\def\ltcalc{ <_\calcsub }
\def\geqcalc{ \geq_\calcsub }

\def\snd{standard normal variable\xspace}
\def\snds{standard normal variables\xspace}

\def\uuv{uniform random variable\xspace}
\def\uuvs{uniform random variables\xspace}

\usepackage[pdftex,hypertexnames=false,hidelinks]{hyperref}

\author{\itshape\Large Pierre Lairez}
\title{\LARGE
A deterministic algorithm to compute\\ approximate roots of polynomial systems\\ in polynomial average time
}
\date{}

\begin{document}

\maketitle

\blfootnote{\scshape Technische Universität Berlin, Germany --- DFG research grant BU 1371/2-2}
\blfootnote{\emph{Date ---} March 15, 2016.}
\blfootnote{\emph{DOI ---} \href{http://dx.doi.org/10.1007/s10208-016-9319-7}{10.1007/s10208-016-9319-7}} 
\blfootnote{
\\[-.5em] }

\vspace{-2\baselineskip}
\begin{abstract}\noindent
  We describe a deterministic algorithm that computes an approximate root
  of~$n$ complex polynomial equations in~$n$ unknowns in average polynomial
  time with respect to the size of the input, in the Blum-Shub-Smale model with
  square root.  It rests upon a derandomization of an algorithm of Beltrán
  and Pardo and gives a deterministic affirmative answer to Smale's
  17{th} problem. The main idea is to make use of the randomness
  contained in the input itself.
\end{abstract}

\section*{Introduction}

Shub and
Smale
provided an extensive theory of Newton's iteration and homotopy continuation
which aims at studying the complexity of computing approximate roots of
complex polynomial systems of equations with as many unknowns as equations.
\autocites{Sma86}{ShuSma96,ShuSma94,ShuSma93,ShuSma93b}{Shu09}
In
their theory, an \emph{approximate root} of a polynomial system refers to a
point from which Newton's iteration converges quadratically to an exact zero of
the system---see Definition~\ref{def:approx}.
This article answers with a deterministic algorithm the following question that they left open:
\begin{question}[{Smale\footcite[17{th} problem]{Sma98}}]
  Can a zero of~$n$ complex polynomial equations in~$n$ unknowns be found
  approximately, on the average, in polynomial time with a uniform algorithm?
\end{question}
The term \emph{algorithm} refers to a machine \emph{à la}
Blum-Shub-Smale\footcite{BluShuSma89} (BSS): a random access memory machine whose
registers can store arbitrary real numbers, that can compute elementary
arithmetic operations in the real field at unit cost and that can branch
according to the sign of a given register.  To avoid vain technical
argumentation, I consider the BSS model extended with the possibility of
computing the square root of a positive real number at unit cost.
The wording \emph{uniform
algorithm} emphasizes the requirement that a single finite machine should solve
all the polynomial systems whatever the degree or the dimension.  The
complexity should be measured with respect to the size of the input, that is
the number of real coefficients in a dense representation of the system to be
solved.  An important characteristic of a root of a polynomial system is its
\emph{conditioning}.  Because of the feeling that approximating a root with
arbitrarily large condition number requires arbitrarily many steps, the problem
only asks for a complexity that is polynomial \emph{on the average} when the
input is supposed to be sampled from a certain probability distribution that we
choose.  The relevance of the average-case complexity is arguable, for the
input distribution may not reflect actual inputs arising from applications.
But yet, average-case complexity sets a mark with which any other result should
be compared.

The problem of solving polynomial systems is a matter of numerical analysis
just as much as it is a matter of symbolic computation. Nevertheless, the
reaches of these approaches differ in a fundamental way.  In an exact setting,
having one root of a generic polynomial system is having them all because of
Galois' indeterminacy, and it turns out that the number of solutions of a
generic polynomial system is the product of the degrees of the equations,
Bézout's bound, and is not polynomially bounded by the number of coefficients
in the input.  This is why achieving a polynomial complexity is only possible
in a numerical setting.

The main numerical method to solve a polynomial system~$f$ is homotopy
continuation.  The principle is to start from another polynomial system~$g$ of
which we know a root~$\eta$ and to move~$g$ toward~$f$ step by step while
tracking all the way to~$f$ an approximate root of the deformed system by
Newton's iteration. The choice of the step size and the complexity of this
procedure is well understood in terms of the condition number along the
homotopy path\autocite{Shu09,BelPar11,BurCuc11}. 
Most of the theory so far is
exposed in the book \emph{Condition}\autocite{BurCuc13}.  The main difficulty
is to choose the starting pair~$(g,\eta)$.  Shub and Smale\autocite{ShuSma94}
showed that there exists good starting pairs, and even many, for some measure, but without providing a way to
compute them efficiently.  Beltrán and Pardo\autocite{BelPar09a,BelPar11}
discovered how to pick a starting pair at random and showed that, on average,
this is a good choice.  This led to a nondeterministic polynomial average-time
algorithm which answers Smale's question.  Bürgisser and
Cucker\footcite{BurCuc11} performed a smoothed analysis of the Beltrán-Pardo
algorithm and described a deterministic algorithm with
complexity~$N^{\cO(\log\log N)}$, where~$N$ is the input size.  The question of
the existence of a deterministic algorithm with polynomial average complexity
it still considered open.

\bigskip
This work provides, with Theorem~\ref{thm:main}, a complete deterministic
answer to Smale's problem, even though, as we will see, it  enriches the
theory of homotopy continuation itself only marginally.  The answer is based on
a derandomization of the nondeterministic Beltrán and Pardo's algorithm
according to two basic observations. Firstly, an approximate root of a
system~$f$ is also an approximate root of a slight perturbation of~$f$.
Therefore, to compute an approximate root of~$f$, one can only consider the
most significant digits of the coefficients of~$f$.  Secondly, the remaining
least significant digits, or noise, of a continuous random variable are
practically independent from the most significant digits and almost uniformly
distributed.  In the BSS model, where the input is given with infinite
precision, this noise can be extracted and can be used in place of a genuine
source of randomness.  This answer shows that for Smale's problem, the
deterministic model and the nondeterministic are essentially equivalent:
randomness is part of the question from its very formulation asking for an
average analysis.  It is worth noting that the idea that the input is subject
to a random noise that does not affect the result is what makes the smoothed analysis of
algorithms relevant\autocite{SpiTen01}.  Also, the study of the resolution of a
system~$f$ given only the most significant digits of~$f$ is somewhat related to
recent works in the setting of machines with finite
precision\autocite{BriCucPen14}.

The derandomization proposed here is different in nature from the
derandomization theorem $\mop{BPP}_\bR =
\mop{P}_\bR$\autocite[\S17.6]{BluCucShu98}, which states that a decision
problem that can be solved over the reals in polynomial time (worst-case
complexity) with randomization and bounded error probability can also be solved
deterministically in polynomial time. Contrary to this work, the
derandomization theorem above relies on the ability of a BSS machine to hold
arbitrary constants in its definition, even hardly computable ones or worse,
not computable ones which may lead to unlikely statements. For example, one can
decide the termination of Turing machines with a BSS machine insofar
Chaitin's~$\Omega$ constant is built in the machine.

\paragraph*{Acknowledgment} I am very grateful to Peter Bürgisser for his help
and constant support, and to Carlos Beltrán for having carefully commented this
work.  I thank the two referees for their meticulous reading and their
insightful suggestions.

\tableofcontents

\section{The method of homotopy continuation}

This part exposes the principles of Newton's iterations and homotopy
continuation upon which rests Beltrán and Pardo's algorithm.  It mostly
contains known results and variations of known results that will be used in the
next part ;  notable novelties are the inequality relating the maximum of the
condition number along a homotopy path by the integral of the cube of the
condition number (Proposition~\ref{lem:maxmu}) and a variant of Beltràn and
Pardo's randomization procedure (Theorem~\ref{thm:bprand}).  For Smale's
problem, the affine setting and the projective setting are known to be
equivalent\autocite{BelPar09a}, so we only focus on the latter.

\subsection{Approximate root}
\label{sec:approxroot}

Let~$n$ be a positive integer.\marginpar{\raggedright\footnotesize (Symbols in the margin mark the place where they are defined.)}
The space~$\bC^{n+1}$ is endowed with the usual
Hermitian inner product.  For~$d\in\bN$, let~$H_d$ denote the vector space of
homogeneous polynomials of degree~$d$ in the variables~$x_0,\dotsc,x_{n}$.
It is endowed with an Hermitian inner product,
called \emph{Weyl's inner product}, for which the monomial basis is an
orthogonal basis and~$\| x_0^{a_0} \dotsm x_n^{a_n} \|^2 = \frac{a_0!\dotsm
a_n!}{(a_1+\dotsb+a_n)!}$.
Let~$d_1,\dotsc,d_n$ be positive integers and let~$\polsys$\marginpar{$\polsys$}
denote~$H_{d_1}\times\dotsb\times H_{d_n}$, the space of all systems of
homogeneous equations in~$n+1$ variables and of degree~$d_1,\dotsc,d_n$.  This
space is endowed with the Hermitian inner product induced by the inner product
of each factor.  The dimension~$n$ and the~$d_i$'s are fixed throughout this
article.  Let~$D$\marginpar{$D$} be the maximum of all~$d_i$'s and let~$N$\marginpar{$N$} denote the complex
dimension of~$\cH$, namely
\[ N = \binom{n+d_1}{n}+\dotsb+ \binom{n+d_n}{n}. \]
Elements of~$\polsys$ are polynomial systems to be solved, and~$2N$ is the
\emph{input size}.  Note that~$2\leq N$, $n^2 \leq N$ and~$D \leq N$.

For every Hermitian space~$V$, we endow the set~$\bS(V)$ of elements of norm~$1$
with the induced Riemannian metric~$d_\bS$:
the distance between two points~$x,y\in\bS(V)$ is the angle between them,
namely~$\cos d_\bS(x,y) = \Re\langle x,y\rangle$.
The projective space~$\bP(V)$ is endowed with the quotient
Riemannian metric~$d_\bP$ defined by
\[ d_{\bP}([x],[y]) \eqdef \min_{\lambda\in\bS(\bC)} d_{\bS}(x,\lambda y). \]

An element of~$f\in\polsys$ is regarded as a homogeneous polynomial
function~$\bC^{n+1}\to\bC^n$.  A \emph{root}---or \emph{solution}, or
\emph{zero}---of~$f$ is a point~$\zeta\in\bP^n$ such that~$f(\zeta)=0$.
Let~$V$\marginpar{$V$} be the \emph{solution variety} $\left\{ (f,\zeta)\in \polsys \times \bP^n \st f(z)=0 \right\}$.
For~$z\in\bC^{n+1}\setminus \left\{ 0 \right\}$, let~$\ud f(z) :
\bC^{n+1}\to\bC^n$ denote the differential of~$f$ at~$z$.  Let~$z^\perp$ be the
orthogonal complement of~$\bC z$ in~$\bC^{n+1}$.  If the restriction~$\ud f(z)|_{z^\perp} :
z^\perp \to \bC^n$ is invertible, we define the \emph{projective Newton operator~$\cN$}\marginpar{$\cN(f,z)$}, introduced by Shub\autocite{Shu93}, by 
\[ \cN(f,z) \eqdef z - \ud f(z)|^{-1}_{z^\perp}(f(z)). \]
It is clear that~$\cN(f,\lambda z)=\lambda \cN(f,z)$, so~$\cN(f,-)$ defines a partial function~$\bP^n \to \bP^n$.

\begin{definition}\label{def:approx}
  A point~$z\in\bP^n$ is an \emph{approximate root} of~$f$ if the sequence
  defined recursively by~$z_0=z$ and~$z_{k+1} = \cN(f,z_k)$ is well defined and if there
  exists~$\zeta\in\bP^n$ such that~$f(\zeta)=0$ and $d_\bP(z_k,\zeta) \leq
  2^{1-2^k} d_\bP(z,\zeta)$ for all~$k\geq 0$.  The point~$\zeta$ is the
  \emph{associated root} of~$z$ and we say that~$z$ \emph{approximates~$\zeta$
  as a root of~$f$.}
\end{definition}

For~$f\in \polsys$ and~$z\in\bC^{n+1}\setminus \left\{ 0 \right\}$, we consider
the linear map
\[ \Xi(f,z) : (u_1,\dotsc,u_n) \in \bC^{n} \mapsto \ud f(z)|^{-1}_{z^\perp}\left( \sqrt{d_1} \|z\|^{d_1-1} u_1,\dotsc, \sqrt{d_n} \|z\|^{d_n-1} u_n \right) \in z^\perp  \]
and the \emph{condition number\footcites{ShuSma93b}[see also][\S16, for more details
about the condition number.]{BurCuc13} of~$f$ at~$z$} is defined to be~$\mu(f,z) \eqdef \| f
\| \, \| \Xi(f,z) \|$,\marginpar{$\mu(f,z)$} where~$\| \Xi(f,z) \|$ is the operator norm.
  When~$\ud
f(z)|_{z^\perp}$ is not invertible, we set~$\mu(f,z)=\infty$.
The condition number is often denoted~$\mu_\text{norm}$ but we stick here to the shorter notation~$\mu$.
For
all~$u,v\in\bC^\times$ we check that~$\mu(u f,v z) = \mu(f,z)$.
We note also that~$\mu(f,z)\geq \sqrt{n} \geq 1$.\footcite[Lemma~16.44]{BurCuc11}
The projective $\mu$-theorem (a weaker form of the better known projective~$\gamma$-theorem) relates the condition number and the notion of approximate root:
\begin{theorem}[{Shub, Smale\footcite{ShuSma93b}}]
  \label{thm:gthm}
  For any~$(f,\zeta)\in V$ and~$z\in\bP^n$, if~$D^{3/2}\mu(f,\zeta) d_\bP(z, \zeta)
  \leq \frac13$, then~$z$ is an approximate root of~$f$ with associated
  root~$\zeta$.
\end{theorem}
\begin{remark}
  The classical form of the result\autocite[\S14, Theorems~1 and~2]{BluCucShu98},
  requires~$D^{3/2}\mu(f,\zeta) \tan(d_\bP(z, \zeta)) \leq 3 - \sqrt{7}$.
  The hypothesis required here is stronger: since~$D^{3/2}\mu(f,\zeta) \geq 1$, if~$D^{3/2}\mu(f,\zeta) d_\bP(z, \zeta)
  \leq \frac13$ then~$d_\bP(z, \zeta) \leq \frac13$ and then
  $\tan(d_\bP(z, \zeta)) \leq 3\tan(\tfrac13) d_\bP(z, \zeta) \leq \frac{3 - \sqrt{7}}{D^{3/2}\mu(f,\zeta)}$
  because~$\tan(\frac13) \leqcalc 3-\sqrt{7}$.
  The symbol~$\leqcalc$\marginpar{$\leqcalc$} indicates an inequality that is easily checked using a calculator.
\end{remark}

The algorithmic use of the condition number heavily relies on this explicit Lipschitz estimate:
\begin{proposition}[{Shub\footcites[Theorem~1]{Shu09}[see also][Theorem~16.2]{BurCuc13}}]
  \label{prop:lipcond}
  Let~$0\leq\varepsilon\leq \frac17$.
  For any~$f,g\in\bP(\cH)$ and~$x,y\in\bP^n$, if
  \[ \mu(f,x) \max\left( D^{1/2} d_\bP(f,g), D^{3/2} d_{\bP}(x,y) \right) \leq \frac{\varepsilon}{4} \]
  then~$(1+\varepsilon)^{-1} \mu(f,x) \leq \mu(g,y) \leq (1+\varepsilon)\mu(f,x)$.
\end{proposition}

\subsection{Homotopy continuation algorithm}
\label{sec:hc-algo}

Let~$I\subset \bR$ be an interval containing~$0$ and let~$t\in I\mapsto
f_t\in\bP(\cH)$ be a continuous function.  Let~$\zeta$ be a root of~$f_0$ such that~$\ud f_0(\zeta)_{|\zeta^\perp}$ is invertible.
There is a
subinterval~$J\subset I$ containing~$0$ and open in~$I$, and a continuous
function~$t\in J\mapsto\zeta_t\in\bP^n$ such that~$\zeta_0 = \zeta$
and~$f_t(\zeta_t)=0$ for all~$t\in J$. We choose~$J$ to be the largest such
interval.
\begin{lemma}
  If~$t\mapsto f_t$ is $C^1$ on~$I$ and if~$\mu(f_t,\zeta_t)$ is bounded on~$J$, then~$J=I$.
  \label{lem:maxhomcont}
\end{lemma}

\begin{proof}
  Without loss of generality, we may assume that~$I$ is compact, so that~$\|\dot f_t\|$ is bounded on~$I$.
  Let~$M$ be the supremum of~$\mu(f_t,\zeta_t) \|\dot f_t\|$ on~$J$.  From
  the construction of~$\zeta_t$ with the implicit function theorem we see
  that~$t\in J\mapsto \zeta_t$ is~$M$-Lipschitz continuous. Hence the
  map~$t\in J\mapsto \zeta_t$ extends to a continuous map on~$\overline J$.
  Thus~$J$ is closed in~$I$,
  and~$I=J$ because~$J$ is also open.
\end{proof}

\begin{proposition}
  \label{prop:condroot}
  Let~$(f,\zeta)\in V$, $g\in \bP(\cH)$ and~$0< \varepsilon \leq \frac17$.
  If~$D^{3/2}\mu(f,\zeta)^2 d_\bP(f,g) \leq
  \frac{\varepsilon}{4(1+\varepsilon)}$, then:
  \begin{enumerate}[(i)]
    \item\label{it:condroot:ex} there exists a unique root~$\eta$ of~$g$ such that $d_\bP(\zeta,\eta) \leq (1+\varepsilon) \mu(f,\zeta) d_\bP(f,g)$;
    \item\label{it:condroot:mu} $(1+\varepsilon)^{-1} \mu(f,\zeta) \leq \mu(g,\eta) \leq (1+\varepsilon) \mu(f,\zeta)$;
    \item\label{it:condroot:app} $\zeta$ approximates~$\eta$ as a root of~$g$ and $\eta$ approximates~$\zeta$ as a root of~$f$.
  \end{enumerate}
\end{proposition}

\begin{proof}
  Let~$t\in[0,1]\mapsto f_t\in\bP(\cH)$ be a geodesic path such that~$f_0 =f$,
  $f_1=g$ and~$\| \dot f_t \| = d_\bP(f,g)$.  Let~$t\in J\mapsto\zeta_t$ be the
  homotopy continuation associated to this path starting from the root~$\zeta$
  and defined as above on a maximal interval~$J\subset [0,1]$.  Let~$\mu_t$
  denote~$\mu(f_t,\zeta_t)$.

  For all~$t\in J$ we know that~$\|\dot\zeta_t\| \leq \mu_t \| \dot f_t
  \|$,\footcite[Corollary~16.14 and Inequality~(16.12)]{BurCuc13} so that 
  \begin{equation}
    d_\bP(\zeta_0,\zeta_t) \leq \int_0^t \|\dot\zeta_u\| du \leq d_\bP(f,g) \int_0^t \mu_u du.
    \label{eqn:proof:condroot}
  \end{equation}

  Let~$J'$ be the closed subinterval of~$J$ defined by~$J'=\left\{ t\in J \st \forall t' \leq
  t, D^{3/2} \mu_0 d_\bP(\zeta_0,\zeta_{t'}) \leq \frac\varepsilon4 \right\}$.
  For all~$t\in J'$ we have~$D^{3/2} \mu_0
  d_\bP(\zeta_0,\zeta_t) \leq \frac\varepsilon4$, by definition, and~$D^{1/2}\mu_0
  d_\bP(f_0,f_t) \leq D^{3/2}\mu^2_0 d_\bP(f,g) \leq \frac\varepsilon4$, by hypothesis.
  Thus, Proposition~\ref{prop:lipcond} ensures that
  \begin{equation}
    (1+\varepsilon)^{-1} \mu_0 \leq \mu_t \leq (1+\varepsilon)\mu_0\text{, for all~$t\in J'$.}
    \label{eqn:proof:condroot:mu}
  \end{equation}

  Thanks to Inequality~\eqref{eqn:proof:condroot} we conclude
  that~$d_\bP(\zeta_0,\zeta_t) \leq (1+\varepsilon) t\, d_\bP(f,g) \mu_0$, for all~$t\in
  J'$, so that~$D^{3/2} \mu_0 d_\bP(\zeta_0,\zeta_t) \leq \frac{t\varepsilon}{4}$, using the assumption~$D^{3/2}\mu(f,\zeta)^2 d_\bP(f,g) \leq
  \frac{\varepsilon}{4(1+\varepsilon)}$.
  This proves that~$J'$ is open in~$J$.  Since it is also closed, we
  have~$J'=J$.  Since~$\mu_t$ is bounded on~$J'$, by
  Inequality~\eqref{eqn:proof:condroot:mu}, Lemma~\ref{lem:maxhomcont} implies
  that~$J'=J = [0,1]$.  Now, Inequalities~\eqref{eqn:proof:condroot}
  and~\eqref{eqn:proof:condroot:mu} imply that~$d_\bP(\zeta_0,\zeta_1)\leq
  (1+\varepsilon) d_\bP(f,g) \mu_0$. This proves~\ref{it:condroot:ex} and~\ref{it:condroot:mu}
  follows from~\eqref{eqn:proof:condroot:mu} for~$t=1$.

  To prove that~$\eta$ approximates~$\zeta$ as a root of~$f$, it is enough to check that
  \[ D^{3/2} \mu(f,\zeta) d_\bP(\zeta,\eta) \leq (1+\varepsilon) D^{3/2} \mu(f,\zeta)^2 d_\bP(f,g) \leq \frac{\varepsilon}{4} \leqcalc \frac{1}{3}, \]
  by Theorem~\ref{thm:gthm}.
  To prove that~$\zeta$ approximates~$\eta$ as a root of~$g$, we check that
  \[ D^{3/2} \mu(g,\eta) d_\bP(\zeta,\eta) \leq (1+\varepsilon)^2 D^{3/2} \mu(f,\zeta)^2 d_\bP(f,g) \leq \frac{\varepsilon(1+\varepsilon)}{4} \leqcalc \frac{1}{3}. \]
  This proves~\ref{it:condroot:app} and the lemma.
\end{proof}

Throughout this article, let~$\varepsilon=\frac{1}{13}$, $A = \frac{1}{52}$, $B=\frac{1}{101}$ and $B'=\frac{1}{65}$.\marginpar{$A$, $B$, $B'$, $\varepsilon$}
The main result that allows computing a homotopy continuation with discrete
jumps is the following:
\begin{lemma}\label{lem:hcmain}
  For any~$(f,\zeta)\in V$ and~$g\in \polsys$ and for any~$z\in \Proj$,   
  if~${D^{3/2} \mu(f,z)}d_\bP(z,\zeta) \leq {A}$ and~${D^{3/2}\mu(f,z)^2} d_\bP(f,g) \leq {B'}$ then:
  \begin{enumerate}[(i)]
    \item\label{it:zappreta} $z$ is an approximate root of~$g$ with some associated root~$\eta$;
    \item\label{it:condb} $(1+\varepsilon)^{-2} \mu(f, z) \leq \mu(g,\eta) \leq (1+\varepsilon)^2 \mu(f, z)$;
    \item\label{it:droot} $D^{3/2} \mu(g,\eta) d_\bP(z,\eta) \leq \frac{1}{23}$.
  \end{enumerate}
  If moreover~${D^{3/2}\mu(f,z)^2}d_\bP(f,g) \leq {B}$ then:
  \begin{enumerate}[(i),resume]
    \item\label{it:rec} ${D^{3/2} \mu(g,z')} \displaystyle d_\bP(z', \eta) \leq {A}$, where~$z'=\cN(g,z)$.
  \end{enumerate}
\end{lemma}

\begin{proof}
Firstly, we bound $\mu(f,\zeta)$.  Since~$D^{3/2}\mu(f,z) d_\bP(z, \zeta) \leq
A = \frac{\varepsilon}{4}$, Proposition~\ref{prop:lipcond} gives
\[ (1+\varepsilon)^{-1} \mu(f,\zeta) \leq \mu(f,z) \leq (1+\varepsilon)\mu(f,\zeta). \]

Next, we have~$D^{3/2} \mu(f,\zeta)^2 d_\bP(f,g) \leq (1+\varepsilon)^2 B' \leqcalc
\frac{\varepsilon}{4(1+\varepsilon)}$, thus Proposition~\ref{prop:condroot}
applies and~$\zeta$ is an approximate root of~$g$ with some associated
root~$\eta$ such that $d_\bP(\zeta,\eta) \leq (1+\varepsilon) \mu(f,\zeta)
d_\bP(f,g)$ and~$(1+\varepsilon)^{-1} \mu(f,\zeta) \leq \mu(g,\eta) \leq
(1+\varepsilon)\mu(f,\zeta)$ and this gives~\ref{it:condb}.

Then, we check that~$z$ approximates~$\eta$ as a root of~$g$. Indeed
\[ d_\bP(z,\eta) \leq d_\bP(z,\zeta)+d_\bP(\zeta,\eta)
  \leq \frac{A+(1+\varepsilon)^2 B'}{D^{3/2} \mu(f,z)}
  \leq  \frac{(1+\varepsilon)^2(A+(1+\varepsilon)^2 B')}{D^{3/2}\mu(g,\eta)}. \]
And~$(1+\varepsilon)^2(A+(1+\varepsilon)^2 B') \leqcalc \frac{1}{23} < \frac13$, so
Theorem~\ref{thm:gthm} applies and we obtain~\ref{it:zappreta}
and~\ref{it:droot}.

We assume now that~${D^{3/2}\mu(f,z)^2}d_\bP(f,g) \leq {B}$. All the
inequalities above are valid with~$B'$ replaced by~$B$.  By definition of an
approximate root~$d_\bP(z',\eta)
\leq \frac{1}{2} d_\bP(z,\eta)$, so that
\[ D^{3/2}\mu(g,\eta) d_\bP(z',\eta) \leq \frac12 (1+\varepsilon)^2(A+(1+\varepsilon)^2 B) \leqcalc \frac{\varepsilon}{4}. \]
Thus~$(1+\varepsilon)^{-1} \mu(g,\eta) \leq \mu(g,z') \leq (1+\varepsilon)\mu(g,\eta)$.

To conclude, we have~$D^{3/2}\mu(g,z') d(z',\eta) \leq \frac12 (1+\varepsilon)^3(A+(1+\varepsilon)^2 B) \leqcalc A$.
\end{proof}
Let~$f,g\in\bS(\cH)$, with~$f\neq -g$.
Let~$t\in[0,1]\mapsto \Gamma(g,f,t)$\marginpar{$\Gamma(g,f,t)$} be the geodesic path from~$g$ to~$f$ in~$\bS(\cH)$.
The condition~$f\neq -g$ guarantees that the geodesic path is uniquely determined,
namely
\begin{equation}
  \Gamma(g,f,t) = \frac{\sin\left( (1-t)\alpha \right)}{\sin(\alpha)} g + \frac{\sin(t\alpha)}{\sin(\alpha)} f,
  \label{eqn:defG}
\end{equation}
where~$\alpha = d_\bS(f,g) \in [0,\pi)$ is the angle between~$f$ and~$g$.

Let~$z\in\bP^n$ such that~${D^{3/2} \mu(g,z)}d_{\bP}(z,\eta) \leq {A}$, for
some root~$\eta$ of~$g$.  By Lemma~\ref{lem:hcmain}\ref{it:zappreta}, applied
with~$g=f$ and~$\eta=\zeta$, the point~$z$ is an approximate root of~$g$, with
associated root~$\eta$.  Given~$g$ and~$z$, we can compute an approximate root
of~$f$ in the following way.  Let~$g_0 = g$, $t_0=0$ and by induction on~$k$ we
define
\[ \mu_k = \mu(g_k,z_k),\ t_{k+1} = t_k + \frac{B}{D^{3/2} \mu_k^{2} d_\bS(f,g)}, \ g_{k+1} = \Gamma(g, f, t_{k+1})\ \text{and}\ 
  z_{k+1} = \cN(g_{k+1}, z_k). \]
Let~$K(f,g,z)$,\marginpar{$K(f,g,z)$} or simply~$K$, be the least integer such that~$t_{K+1} > 1$, if
any, and~$K(f,g,z)=\infty$ otherwise.  Let~$\tilde M(f,g,z)$\marginpar{$\tilde M(f,g,z)$} denote the maximum
of all~$\mu_k$ with $0\leq k\leq K$.  Let~$\mop{HC}$\marginpar{$\mop{HC}(f,g,z)$} be the procedure
that takes as input~$f$, $g$ and~$z$ and outputs~$z_K$.
Algorithm~\ref{algo:hc} recapitulates the definition. It terminates if and only
if~$K<\infty$, in which case~$K$ is the number of iterations.
For simplicity, we assume that we can compute exactly the
square root function, the trigonometric functions and the operator norm
required for the computation of~$\mu(f,z)$.  Section~\ref{sec:bss} shows how
to implement things in the BSS model extended with the square root only.

\begin{algo}[tp]
  \centering
  \begin{algorithmic}
    \Function{HC}{$f$, $g$, $z$}
    \State $t\gets {1}/\left(101 D^{3/2} \mu(g, z)^{2} d_\bS(f,g)\right)$
      \While{$1 > t$}
        \State $h \gets \Gamma(g, f, t)$
        \State $z \gets \cN(h, z)$
        \State $t \gets t + {1}/\left(101 D^{3/2} \mu(h, z)^{2} d_\bS(f,g)\right)$        
      \EndWhile
      \State \textbf{return} $z$
    \EndFunction
  \end{algorithmic}
  \caption[]{Homotopy continuation
    \begin{description}
      \item[Input.] $f$, $g \in \bS(\cH)$ and $z\in\bP^n$.
      \item[Precondition.] There exists a root~$\eta$ of~$g$ such that~${52\,D^{3/2}
      \mu(g,z)} d_{\bP}(z,\eta) \leq {1}$.
      \item[Output.] $w\in\bP^n$
      \item[Postcondition.] $w$ is an approximate root of~$f$.
    \end{description}
  }
  \label{algo:hc}
\end{algo}

Let~$h_t = \Gamma(f,g,t)$ and let~$t\in J\mapsto \zeta_t$ be the homotopy
continuation associated to~$t\in[0,1]\mapsto h_t$, where~$\eta_0$ is the
associated root of~$z$, defined on a maximal subinterval~$J\subset[0,1]$.  Let\marginpar{$I_p(f,g,z)$, $M(f,g,z)$}
\begin{equation}
 M(f,g,z) \eqdef \sup_{t\in J} \mu(f_t,\zeta_t) \quad \text{and}\quad I_p(f,g,z)\eqdef  \int_J \mu(h_t,\eta_t)^p \ud t.  
  \label{eqn:def-M-Ip}
\end{equation}

The behavior of the procedure~$\mop{HC}$ can be controlled in terms of the
integrals~$I_p(f,g,z)$. It is one of the corner stone of the complexity theory
of homotopy continuation methods.
The following estimation of the maximum of the condition number, along a homotopy path,
in terms of the third moment of the condition number seems to be original.
It will be important for the average complexity analysis.

\begin{proposition}
  If~$J=[0,1]$ then~$M(f,g,z) \leq 151\, D^{3/2} I_3(f,g,z)$.
  \label{lem:maxmu}
\end{proposition}

\begin{proof}
  Let~$\varepsilon=\frac{1}{7}$ and let~$s\in [0,1]$ such
  that~$\mu(f_s,\zeta_s)$ is maximal.  For all~$t\in [0,1]$, $d_\bS(f_s,f_t)
  \leq |t-s| d_\bS(f,g)$.  Thus, if
  \begin{equation}
    |t-s| \leq \frac{\varepsilon}{4(1+\varepsilon) D^{3/2} \mu(f_s,\zeta_s)^2 d_\bS(f,g)},
    \label{eqn:proof:maxmu}
  \end{equation}
  then~$\mu(f_t,\zeta_t) \geq (1+\varepsilon)^{-1} \mu(f_s,\zeta_s)$, by
  Proposition~\ref{prop:condroot}.  Since~$d_\bS(f,g) \leq \pi$, the diameter
  of the interval~$H$ of all~$t\in [0,1]$ satisfying
  Inequality~\eqref{eqn:proof:maxmu} is at least~$\frac{\varepsilon}{4\pi
    (1+\varepsilon) D^{3/2} \mu(f_s,\zeta_s)^2}$.  Thus
  \[ \int_0^1 \mu(f_t,\zeta_t)^3 \ud t \geq \int_H \frac{\mu(f_s,\zeta_s)^3}{(1+\varepsilon)^3} \ud t \geq \frac{\varepsilon\, \mu(f_s,\zeta_s)}{4 \pi (1+\varepsilon)^4 D^{3/2}} \geqcalc \frac{1}{151} \frac{\mu(f_s,\zeta_s)}{D^{3/2}}. \qedhere \]
\end{proof}

\begin{theorem}[{Shub\footcite{Shu09}}]
  \label{thm:complexity-hc}
  With the notations above,  if~${D^{3/2} \mu(g,z)} d_{\bP}(z,\eta) \leq {A}$ then:
  \begin{enumerate}[(i)]
    \item \label{it:Kinfty} $\mop{HC}(f,g,z)$ terminates if and only if~$I_2(f,g,z)$ is finite, in which case~$J=[0,1]$;
  \end{enumerate}
  If moreover $\mop{HC}(f,g,z)$ terminates then:
  \begin{enumerate}[(i),resume]
    \item \label{it:Mtilde} $(1+\varepsilon)^{-2} M(f,g,z) \leq \tilde M(f,g,z) \leq (1+\varepsilon)^2 M(f,g,z)$.
    \item \label{it:boundK} $K(f,g,z) \leq 136\, D^{3/2} d_{\bS}(f,g) I_2(f,g,z)$;
    \item \label{it:HCapprox} $\mop{HC}(f,g,z)$ is an approximate root of~$f$;
    \item \label{it:HCapproxb} $D^{3/2}\mu(f,\zeta)d_\bP(\mop{HC}(f,g,z),\zeta) \leq \frac{1}{23}$, where~$\zeta$ is the associated root of~$\mop{HC}(f,g,z)$.
  \end{enumerate}
\end{theorem}

\begin{proof}
  Let~$\eta_k$ denote~$\zeta_{t_k}$.
  Since~${D^{3/2} \mu_k^{2}}d_\bP(g_k,g_{k+1}) \leq {B}$ for
  all~$k\geq 0$, Lemma~\ref{lem:hcmain}\ref{it:rec} proves, by induction
  on~$k$  that~${D^{3/2} \mu_k}d_\bP(z_k,\eta_{k}) \leq {A}$ for any~$k\geq 0$.
 
  Assume that~$[0,t_k]\subset J$ for some~$k\geq 0$ and let~$t\in
  [t_k,t_{k+1}]\cap J$ so that
  \[ D^{3/2} \mu_k^2 d(g_k,h_t) \leq D^{3/2} \mu_k^2 d(g_k,g_{k+1}) \leq {B}. \]
  Because~$D^{3/2} \mu_kd(z_k,\eta_{k}) \leq
  {A}$, Lemma~\ref{lem:hcmain}\ref{it:condb} applies to $(g_k,\eta_{k})$,
  $h_t$ and~$z_k$ and asserts that
  \begin{equation}
    (1+\varepsilon)^{-2} \mu_k\leq\mu(h_t,\zeta_t) \leq (1+\varepsilon)^2 \mu_k.
    \label{eqn:proof:hc:mu}
  \end{equation}
  By definition $\mu_k^2 (t_{k+1}-t_k) = \frac{B}{D^{3/2}d_\bS(f,g)}$, so
  integrating over~$t$ leads to
  \begin{gather}
    \label{eqn:proof:Kint:a}
    \int_{0}^{t_k} \mu(h_t,\zeta_t)^2 \ud t \geq (1+\varepsilon)^{-4}\sum_{j=0}^{k-1} \mu_j^2 (t_{j+1}-t_j)
    = \frac{k B}{(1+\varepsilon)^4 D^{3/2}d_\bS(f,g)},\\
    \label{eqn:proof:Kint:b}    
    \text{and}\quad \int_{0}^{\sup J} \mu(h_t,\zeta_t)^2 \leq (1+\varepsilon)^4 \sum_{j=0}^{k} \mu_j^2 (t_{j+1}-t_j)
= \frac{(1+\varepsilon)^4  (k+1) B}{D^{3/2}d_\bS(f,g)}.
  \end{gather}
    
  Assume now that~$I_2(f,g,z)$ is finite.  The left-hand side of
  Inequality~\eqref{eqn:proof:Kint:a} is finite so there exists a~$k$ such
  that~$t_{k+1}\not\in J$.  But then Inequalities~\eqref{eqn:proof:hc:mu} shows
  that~$\mu_t$ is bounded on~$J$ which implies, Lemma~\ref{lem:maxhomcont}
  that~$J=[0,1]$. And since~$t_{k+1} \not\in J$, this proves that~$K$ is finite.

  Conversely, assume that~$K$ is finite, \emph{i.e.} $\mop{HC}(f,g,z)$ terminates.
  Then there exists a maximal~$k$ such that~$[0,t_k]\subset J$
  and thus for all~$t\in J$
  \[ \mu(h_t,\zeta_t) \leq (1+\varepsilon)^2 \max_{j\leq k} \mu(g_k,z_k). \]
  So~$\mu(h_t,\zeta_t)$ is bounded on~$J$, which implies that~$J=[0,1]$, and
  thus~$k=K$.  Inequality~\eqref{eqn:proof:Kint:b} then shows that~$I_2(f,g,z)$
  is finite, which concludes the proof of~\ref{it:Kinfty}.
  We keep assuming that~$K$ is finite.
  Inequality~\eqref{eqn:proof:hc:mu} shows~\ref{it:Mtilde}.
  Since~$[0,t_K]\subset [0,1]$, by definition,
  Inequalities~\eqref{eqn:proof:Kint:a} and~\eqref{eqn:proof:Kint:b} shows
  that
  \[ \frac{1}{B(1+\varepsilon)^4} D^{3/2} d_{\bS}(f,g) I_2(f,g,z) - 1 \leq  K \leq \frac{(1+\varepsilon)^4}{B} D^{3/2} d_{\bS}(f,g) I_2(f,g,z). \]
  We check that~$\frac{(1+\varepsilon)^4}{B} \leqcalc 136$, which gives~\ref{it:boundK}.
  Finally,
  Lemmas~\ref{lem:hcmain}\ref{it:zappreta}  and~\ref{lem:hcmain}\ref{it:droot}
  show that~$z_K$ approximates~$\zeta_1$ as a root of~$f$ and that $D^{3/2}
  \mu(f,\zeta_1) d_\bP(z_K,\zeta_1) \leq \frac{1}{23}$, which
  gives~\ref{it:HCapprox} and~\ref{it:HCapproxb}.
\end{proof}

\subsection{A variant of Beltrán-Pardo randomization}
\label{sec:bp}

An important discovery of Beltrán and Pardo is a procedure to pick a random
system and one of its root simultaneously without actually solving  any
polynomial system. And from the complexity point of view, it turns out that a
random pair~$(g,\eta)\in V$ is a good starting point to perform the homotopy
continuation.

Let~$g \in \SH$ be a \uuv, where the uniform measure is relative to the
Riemannian metric on~$\SH$.  Almost surely~$g$ has finitely many roots
in~$\bP^n$.  Let~$\eta$ be one of them, randomly chosen with the uniform
distribution.  The probability distribution of the random variable~$(g,\eta)\in
V$ is denoted~$\rstd$.\marginpar{$\rstd$}
The purpose of Beltrán and Pardo's
procedure\footcites[\S2.3]{BelPar11}[see also][Chap.~17]{BurCuc13} 
is to generate random pairs~$(g,\eta)$, according to the distribution~$\rstd$
without solving any polynomial system.
We give here a variant
which
requires only a \uuv in~$\bS(\bC^{N}) \simeq \SH$ as the source of randomness.

Let us assume that~$f = (f_1,\dotsc,f_n)\in \SH$ is a \uuv and
write~$f$ as
\[ f_i = c_i x_0^{d_i} + \sqrt{d_i} x_0^{d_i-1} \sum_{j=1}^n a_{i,j} x_i +
  f'_i(x_{0},\dotsc, x_n), \]
for some~$c_i$, $a_{i,j}\in\bC$ and~$f'_i\in H_{d_i}$ such that~$f'_i(e_0) = 0$ and~$\ud f'_i(e_0) = 0$.
Let~$f'=(f'_{1},\dotsc,f'_n)\in \polsys$. By construction, $f'(e_0) = 0$ and~$\ud f'(e_0) = 0$.
Let
\[ M \eqdef
  \begin{pmatrix}
    a_{1,1} & \dotsm & a_{1,n} & c_1 \\
    \vdots & \ddots & \vdots & \vdots \\
    a_{n,1}& \dotsm & a_{n,n} & c_n
  \end{pmatrix} \in \bC^{n\times (n+1)}.
\]
Almost surely, $\ker M$ has dimension~$1$; Let~$\zeta\in\bP^n$ be the point
representing~$\ker M$ and let~$\zeta'\in\bS(\bC^{n+1})$ be the unique
element of~$\ker M \cap \bS(\bC^{n+1})$ whose first nonzero coordinate is a
real positive number.  Let~$\Psi_{M,\zeta'} =
(\Psi_1,\dotsc,\Psi_n)\in \polsys$ be defined by
\begin{equation}
  \Psi_i \eqdef \sqrt{d_i} \left( \sum_{i=0}^n x_i \overline{\zeta'_i} \right)^{d_i-1}\sum_{j=0}^n m_{i,j} x_j,
  \label{equ:defPsi}
\end{equation}
where~$\overline{\zeta'_i}$ denotes the complex conjugation.
By construction~$\Psi_{M,\zeta'}(\zeta) = 0$.  Let~$u\in U(n+1)$, the unitary group
of~$\bC^{n+1}$, such that~$u(e_0) = \zeta$.  The choice of~$u$ is arbitrary but
should depend only on~$\zeta$.  For example, we can choose~$u$, for almost
all~$\zeta$, to be the unique element of~$U(n+1)$ with determinant~$1$ that is
the identity on the orthogonal complement of~$\left\{ e_0,\zeta \right\}$ and
that sends~$e_0$ to~$\zeta$.  Finally, let~$g = f' \circ u^{-1} + \Psi_{M,\zeta'} \in
\polsys$.  By construction~$g(\zeta) = 0$.  We define~$\mop{BP}(f) \eqdef (g,
\zeta)$\marginpar{$\mop{BP}(f)$} which is a point in the solution
variety~$V$\!.

\begin{theorem}
  If~$f\in\SH$ is a \uuv,
  then~$\mop{BP}(f)\sim\rho_{\text{std}}$.
  \label{thm:bprand}
\end{theorem}

\begin{proof}
  We reduce to another variant of Beltrán-Pardo procedure given by Bürgisser and
  Cucker\footcite[Theorem~17.21(a)]{BurCuc13} in the case of Gaussian
  distributions.  Let~$f\in\SH$ be a \uuv, and let~$\chi \in
  [0,\infty)$ be an independent random variable following the \emph{chi}
  distribution with~$2N$ degrees of freedom, so that~$\chi f$ is a centered
  Gaussian variable in~$\cH$ with covariance matrix~$I_{2N}$ (which we call
  hereafter a \emph{\snd}).  For~$\zeta \in \bP^n$, let~$R_\zeta \subset \cH$ be the
  subspace of all~$g$ such that~$g(\zeta)=0$ and~$\ud g(\zeta) = 0$
  and let~$S_\zeta$ be the orthogonal complement of~$R_\zeta$ in~$\cH$.  The
  system~$\chi f$ splits orthogonally as~$\chi f' + \chi h$, where~$\chi
  f'\in R_{e_0}$  and $\chi h \in S_{e_0}$ are independent \snds.

  Let~$M \in \bC^{n\times (n+1)}$, $\zeta\in \bP^n$, $\zeta'\in\bS^n$
  and~$u\in U(n+1)$ be defined in the same way as in the definition
  of~$\BP(f)$.
  The map that gives~$M$ as a function of~$h$ is an isometry~$S_{e_0} \to
  \bC^{n\times(n+1)}$, so~$\chi M$ is a \snd that is independent from~$f'$.  Let~$\lambda \in
  \bS(\bC)$ be an independent \uuv, so that~$\lambda \zeta'$ is uniformly
  distributed in~$\ker M \cap \bS^{n}$ when~$M$ has full rank, which is the
  case with probability~$1$.
  The composition map~$g\in R_{e_0} \mapsto g\circ u^{-1} \in R_\zeta$ is an
  isometry. Thus, conditionally to~$\zeta$, the system~$\chi f'\circ u^{-1}$ is a
  \snd in~$R_\zeta$.  As a consequence, and according to
  Bürgisser and Cucker\footcite[Theorem~17.21(a)]{BurCuc13}, the system
  \[ H \eqdef \chi f'\circ u^{-1} + \Psi_{\chi M,\lambda \zeta'} \]
  is a \snd in~$\cH$ and~$\zeta$ is uniformly distributed among
  its roots.  Hence~$H/\|H\|$ is uniformly distributed in~$\SH$ and~$\left(
  H/\|H\|, \zeta  \right) \sim \rstd$.

  We check easily that~$\| \Psi_{M,\lambda \zeta'} \| = \|M\|_F = \| h \|$,
  where~$\|M\|_F$ denotes the Froebenius matrix norm, that is the usual
  Hermitian norm on~$\bC^{n\times (n+1)}$.  Moreover~$\|f'\circ u^{-1} \| =
  \|f'\|$, this is the fundamental property of Weyl's inner product on~$\cH$.
  Thus~$\|H\|=\|f\| = \chi$, and in turn
  \[  \left(  f'\circ u^{-1} + \Psi_{ M,\lambda \zeta'},\  \zeta \right) = \left( H/\|H\|,\  \zeta  \right) \sim \rstd, \]
  which is almost what we want, up to the presence of~$\lambda$.
  Let~$\Delta \in \bC^{n\times n}$ be the diagonal matrix given by~$( \bar\lambda^{d_i-1} ){}_{1\leq i\leq n}$.
  It is clear that~$\Psi_{ M,\lambda \zeta'} = \Psi_{ \Delta M,\zeta'}$.
  The map~$M\mapsto \Delta M$ is an isometry of~$\bC^{n \times (n+1)}$ and $\ker \Delta M = \ker M$ so~$(\chi M, u, \zeta')$ and~$( \chi \Delta M, u, \zeta')$ have the same probability distribution.
  Since~$\chi f'$ is independent from~$\chi M$ and~$\lambda$, it follows that the system~$H'$ defined by
  \[ H' \eqdef \chi f'\circ u^{-1} + \Psi_{ \chi M, \zeta'}, \]
  has the same probability distribution as~$H$.
  To conclude the proof, we note that~$\|H'\| = \chi$ and that~$(H' / \chi, \zeta) = \BP(f)$.   
\end{proof}

Given~$f\in \SH$, Beltrán and Pardo's algorithm proceeds in sampling a
system~$g\in\SH$ from the uniform distribution and then
computing~$\mop{HC}(f, \mop{BP}(g))$.  If the input~$f$ is a \uuv then we
can evaluate the expected number of homotopy
steps~$\bE(K(f,\mop{BP}(g)))$.  Indeed, let~$\eta$ be root of~$g$,
uniformly chosen, the theorem above asserts that~$\mop{BP}(g)$ has the same
probability distribution as~$(g,\eta)$ so~$\bE(K(f,\mop{BP}(g))) = \bE(K(f,g,\eta))$.
Thanks to Theorem~\ref{thm:complexity-hc}\ref{it:boundK}, it is not difficult to see that
$\bE(K(f,g,\eta)) \leq 214\, D^{3/2} \bE(\mu(g,\eta)^2)$.
This is why the estimation of~$\bE(\mu(g,\eta)^2)$ is another corner stone of
the average complexity analysis of homotopy methods.  Deriving from a identity
of Beltrán and Pardo, we obtain the following:
\begin{theorem}
  If~$(g,\eta) \sim \rstd$, then $\bE(\mu(g,\eta)^{p}) \leq \frac{3}{4-p}(nN)^{p/2}$ for any~$2\leq p < 4$.
  \label{thm:mumoment}
\end{theorem}

\begin{proof}
  Let~$s = p/2 - 1$. 
  Beltrán and Pardo\footcite[Theorem~23]{BelPar11} state that
  \[ \bE(\mu(g,\eta)^{2+2s}) = \frac{\Gamma(N+1)}{\Gamma(N-s)} \sum_{k=1}^{n} \binom{n+1}{k+1} \frac{\Gamma(k-s)}{\Gamma(k)} n^{-k+s}. \]

  We use the 
  inequalities~$x^{-y}\Gamma(x) \leq \Gamma(x-y) \leq (x-1)^{-y} \Gamma(x)$, for~$x\in [1,\infty)$ and~$y\in [0,1]$,
  which comes from the log-convexity of~$\Gamma$.
  In particular, since~$0\leq s < 1$,
  \[ \frac{\Gamma(N+1)}{\Gamma(N-s)} = \frac{N \Gamma(N)}{\Gamma(N-s)} \leq N^{1+s} \quad\text{and}\quad \frac{\Gamma(k-s)}{\Gamma(k)} \leq (k-1)^{-s}.  \]
  Thus
  \[ \bE(\mu(g,\eta)^{2+2s}) \leq N^{1+s} \left( \binom{n+1}{2} \frac{\Gamma(1-s)}{\Gamma(1)} n^{s-1} + \sum_{k=2}^{n} \binom{n+1}{k+1} {(k-1)^{-s}} n^{-k+s} \right).
    \]
  On the one hand~$(1-s)\Gamma(1-s)= \Gamma(2-s) \leq \Gamma(2) = \Gamma(1)$, so
  \[ \binom{n+1}{2} \frac{\Gamma(1-s)}{\Gamma(1)} n^{s-1} \leq \frac{(n+1)n}{2} \frac{1}{1-s} n^{s-1} \leq \frac{n^{1+s}}{1-s}. \]
  On the other hand, 
  \begin{align*}
    \sum_{k=2}^{n} \binom{n+1}{k+1} {(k-1)^{-s}} n^{-k+s} &\leq n^{1+s} \sum_{k=3}^{n+1} \binom{n+1}{k} n^{-k} \\
    &=  n^{1+s} \left( \left( 1+\frac{1}{n} \right)^{n+1} - 1 - \frac{n+1}{n} - \frac{1}{n^2}\binom{n+1}{2} \right) \\
    &\leqcalc  \frac{n^{1+s}}{4} \leq \frac{n^{1+s}}{4(1-s)}.
  \end{align*}
  Putting together all above, we obtain the claim.
\end{proof}

\section{Derandomization of the Beltrán-Pardo algorithm}

\subsection{Duplication of the uniform distribution on the sphere}
\label{sec:noise}

An important argument of the construction is the ability to produce
approximations of two independent \uuvs in~$\bS^{2N-1}$ from a single \uuv
in~$\bS^{2N-1}$ given with infinite precision.  More precisely, let~$Q$ be a
positive integer.  This section is dedicated to the construction of two
functions~$\flo{-}_Q$ and~$\{-\}_Q$ from the sphere~$\bS^{2N-1}$ to itself,
respectively called the \emph{truncation} and the \emph{fractional part} at
precision~$Q$.  For~$u\in\bS^{2N-1}$, $\flo{u}_Q$ is close to~$u$ and if~$u$ is
\uuv, then~$\{u\}_Q$ is \emph{nearly} uniformly distributed
in~$\bS^{2N-1}$ and \emph{nearly} independent from~$\flo{u}_Q$, in the following
sense:
\begin{lemma}
  \label{lem:qindepA}
  For any~$u\in \bS^{2N-1}$, $d_\bS\left(\flo{u}_Q, u\right) \leq 3 N^{1/2}/Q$.
  Moreover, for any continuous nonnegative function~$\Theta : \bS^{2N-1} \times \bS^{2N-1} \to \bR$,
  \[ \frac{1}{\vol(\bS^{2N-1})} \int_{\bS^{2N-1}} \Theta\left( \flo{u}_Q,\left\{ u \right\}_Q \right) \ud u \leq \frac{\exp\left( \tfrac{2 N^{3/2}}{Q} \right)}{\vol(\bS^{2N-1})^2} \int_{\bS^{2N-1}}\int_{\bS^{2N-1}} \Theta\left( \flo{u}_Q, v \right) \ud u\,\ud v. \]
\end{lemma}

For~$x\in \bR$, let~$A(x)$ denote the integral part of~$a$ and let~$A_Q( a ) \eqdef Q^{-1} A( Q a )$
be the truncation at precision~$Q$.
For~$x\in \bR^{2N-1}$, let~$A_Q(x) \in \bR^{2N-1}$ be the
vector~$(A_Q(x_1),\dotsc,A_Q(x_{2N-1}))$ and let~$B_Q(x) \eqdef (x-A_Q(x))Q$, which is
a vector in~$[0,1]^{2N-1}$.  We note that~$\|A_Q(x)-x\|^2 \leq (2N-1)/Q^2$,
because the difference is bounded componentwise by~$1/Q$.

Let~$C$ and~$C_+$ denote~$[-1,1)^{2N-1}$ and~$[0,1)^{2N-1}$ respectively, and let~$F(x) =
  (1+\|x\|^2)^{-N}$.  We first show that if~$x\in C$ is a random variable with
  probability density function~$F$ (divided by the appropriate normalization
  constant) then~$B_Q(x)$ is \emph{nearly} uniformly distributed in~$C_+$ and
  \emph{nearly} independent from~$A_Q(x)$.

\begin{lemma}\label{lem:qindepC}
  For any continuous nonnegative function~$\Theta : [-1,1]^{2N-1}\times [0,1]^{2N-1} \to \bR$,
  \[ \int_{C} \Theta\left( A_Q(x), B_Q(x) \right) F(x)\ud x \leq \exp\left( \tfrac{2 N^{3/2}}{Q} \right) \int_{C_+}\int_{C} \Theta\left( A_Q(x), y \right) F(x)\ud x\,\ud y. \]
\end{lemma}

\begin{proof}
  For any integers~$-Q \leq k_i < Q$, for~$1\leq i\leq 2N-1$, the
  function~$A_Q$ is constant on the set~$\prod_{i=1}^{2N-1}\left[
  \frac{k_i}{Q}, \frac{k_i+1}{Q} \right)$, and these sets form a partition
  of~$C$.  Let~$X_1,\dotsc,X_{(2Q)^{2N-1}}$ denote an enumeration
  of these sets and let~$a_k$ denote the unique value of~$A_Q$ on~$X_k$.
  The diameter of~$X_k$ is~$\sqrt{2N-1}/Q$.
  Since the function~$x\in[0,\infty) \mapsto -N\log(1+x^2)$ is~$N$-Lipschitz continuous,
  we derive that
  \begin{equation}
    \max_{X_k} F \leq e^{N\sqrt{2N-1}/Q}\min_{X_k} F \leq e^{2N^{3/2}/Q }\min_{X_k} F.
    \label{eqn:proof:qindep:a}
  \end{equation}
  For any~$1\leq k\leq (2Q)^{2N-1}$, we have
  \begin{equation*}
    \int_{X_k} \Theta\left( A_Q(x), B_Q(x) \right) F(x) \ud x \leq \max_{X_k} F \int_{X_k} \Theta\left(a_k,(x-a_k)Q\right)\ud x,
  \end{equation*}
  because~$A_Q(x) = a_k$ on~$X_k$ and by definition of~$B_Q(x)$.
  A simple change of variable shows that
  \[ \int_{X_k} \Theta\left(a_k,(x-a_k)Q\right)\ud x = \vol(X_k) \int_{C_+} \Theta(a_k,y)\ud y. \]
  Besides, for all~$y\in C_+$,
  \[   \Theta(a_k,y) \leq \frac{1}{\vol(X_k)\min_{X_k}F} \int_{X_k} \Theta\left( A_Q(x), y \right) F(x)\ud x. \]
  Putting together all above and summing over~$k$ gives the claim.
\end{proof}

Thanks to a method due to Sibuya, we may transform a uniform random variable of~$C_+$ into a
uniform random variable in~$\bS^{2N-1}$.
Let~$x = (x_1,\dotsc,x_{2N-1})\in C_+$,
let~$y_1,\dotsc,y_{N-1}$ denote the numbers~$x_{N+1},\dotsc,x_{2N-1}$ arranged in ascending order,
and let~$y_0=0$ and~$y_N=1$.
Let~$S(x)\in \bR^{2N}$\marginpar{$S(x)$} denote the vector such that for any~$1\leq i \leq N$
\begin{equation}
 S(x)_{2i-1} \eqdef \sqrt{y_i-y_{i-1}} \cos(2\pi x_i) \quad \text{and}\quad S(x)_{2i} \eqdef \sqrt{y_{i}-y_{i-1}} \sin(2\pi x_i).  
  \label{eqn:sibuya}
\end{equation}
\begin{proposition}[{Sibuya\autocite{Sib62}}]\label{prop:sibuya}
  If~$x$ is a uniformly distributed random variable in~$C_+$, then~$S(x)$ is uniformly distributed in~$\bS^{2N-1}$.
\end{proposition}

We now define~$\flo{-}_Q$ and~$\left\{ - \right\}_Q$.  Let~$\Sigma\subset \bR^{2N}$ be the set of
all~$x\in \bR^{2N}$ such that~$\|x\|_\infty = 1$.  It is divided
into~$4N$ faces that are isometric to~$C$: they are the sets
$\Sigma_i^\varepsilon \eqdef \left\{ x\in\Sigma \st x_i = \varepsilon \right\}$,
for~$\varepsilon\in\left\{ -1,1 \right\}$ and~$1\leq i \leq 2N$
and the isometries are given by the maps
\[ t_{i,\varepsilon} : \Sigma_i^\varepsilon \to C,\quad x\mapsto (x_1,\dotsc,x_{i-1},x_{i+1},\dotsc,x_n). \]
Through these isometries, we define the functions~$A'_Q$ and~$B'_Q$ on~$\Sigma$:
for~$x\in\Sigma_i^\varepsilon$ we set
\[ A'_Q(x) \eqdef t^{-1}_{i,\varepsilon}(A_Q(t_{i,\varepsilon}(x))) \in \Sigma_i^\varepsilon \quad\text{and}\quad B'_Q(x) \eqdef B_Q(t_{i,\varepsilon}(x))\in C_+. \]
Let~$\nu_\infty : u \in\bS^{2N-1} \mapsto u/\|u\|_\infty \in \Sigma$ and its inverse~$\nu_2 : x\in\Sigma \mapsto x/\|x\| \in \bS^{2N-1}$.
Finally, we define, for~$u\in\bS^{2N-1}$, using Sibuya's function~$S$, see Equation~\eqref{eqn:sibuya},\marginpar{$\flo{u}_Q$, $\left\{ u \right\}_Q$}
\begin{equation}
  \flo{u}_Q \eqdef \nu_2\left( A'_Q(\nu_\infty(u)) \right) \quad\text{and}\quad \left\{ u \right\}_Q \eqdef S\left(B'_Q( \nu_\infty(u) )\right).
  \label{eqn:defaQ}
\end{equation}

We may now prove Lemma~\ref{lem:qindepA}.

\begin{proof}[Proof of Lemma~\ref{lem:qindepA}]
  Let~$u\in \bS^{2N-1}$.
  It is well-known that~$d_\bS(\flo{u}_Q,u)\leq \frac{\pi}{2} \| \flo{u}_Q - u\|$.
  Furthermore, the map~$\nu_2$ is clearly~$1$-Lipschitz continuous on~$\Sigma$ so
  $\| \flo{u}_Q - u\| \leq \| A'_Q(\nu_\infty(u)) - \nu_\infty(u) \|$
  and we already remarked that the latter is at most~$\smash{\sqrt{2N-1}}/Q$.
  With~$\frac{\pi}{2} \sqrt{2} \leqcalc 3$, this gives the first claim.

  Concerning the second claim, we consider the partition of the sphere by the sets~$\nu_2(\Sigma_i^\varepsilon)$.
  For any~$u\in \nu_2(\Sigma_i^\varepsilon)$, we have
  \[ \nu_\infty(u) = \left( \frac{u_1}{u_i}, \dotsc, \frac{u_{i-1}}{u_i}, 1, \frac{u_{i+1}}{u_i}, \dotsc, \frac{u_{n}}{u_i} \right). \]
  Hence, by Lemma~\ref{lem:jacproj} below, 
  the absolute value of the Jacobian of the map~$\nu_\infty : \nu_2(\Sigma_i^\varepsilon) \to \Sigma_i^\varepsilon$
  at some~$\nu_2(x)$, $x\in \Sigma_1^\varepsilon$, is precisely
  \[ \left(\nu_2(x)_i \right)^{-2N} = \|x\|^{2N} = \left(1+\|t_{i,\varepsilon}(x)\|^2\right)^{N} = F(t_{i,\varepsilon}(x))^{-1}. \]
  Thus, the change of variable~$\nu_\infty$ gives 
  \begin{align*}
    \int_{\nu_2(\Sigma_i^\varepsilon)} \Theta\left(\flo{u}_Q,\left\{ u \right\}_Q \right) \ud u
    &= \int_{\Sigma_i^\varepsilon} \Theta\left(\nu_2\left( A'_Q(x) \right), S(B'_Q(x)) \right) F( t_{i,\varepsilon}(x)) \ud x,
    \intertext{and then Lemma~\ref{lem:qindepC} (applied over~$\Sigma_i^\varepsilon$ with the isometry~$t_{i,\varepsilon} : \Sigma_i^\varepsilon\to C$) implies}
    &\leq \exp\left( \tfrac{2 N^{3/2}}{Q} \right)\int_{\Sigma_i^\varepsilon}\int_{C_+} \Theta\left(\nu_2\left( A'_Q(x) \right), S(y) \right) \ud y\, F(t_{i,\varepsilon}(x)) \ud x
    \intertext{and Proposition~\ref{prop:sibuya} gives the equality}
    &= \frac{\exp\left( \tfrac{2 N^{3/2}}{Q} \right)}{\vol(\bS^{2N-1})}\int_{\Sigma_i^\varepsilon}\int_{\bS^{2N-1}} \Theta\left(\nu_2\left( A'_Q(x) \right), v \right) \ud v\, F(t_{i,\varepsilon}(x)) \ud x,
  \end{align*}
  and applying the inverse change of variable~$\nu_2 : \Sigma_i^\varepsilon \to \nu_2(\Sigma_i^\varepsilon)$ and summing over~$i$ and~$\varepsilon$ gives the claim.
\end{proof}

\begin{lemma}
  Let~$H = \left\{ x\in\bR^{2N}\st x_1 > 0 \right\}$ and let~$\varphi$ be the map
  \begin{align*}
    \varphi : \bS^{2N-1} \cap H &\longrightarrow \bR^{2N-1} \\
    (u_1,\dotsc,u_{2N}) &\longmapsto \left( \frac{u_2}{u_1}, \dotsc, \frac{u_{2N}}{u_1} \right).
  \end{align*}
  Then, for any~$u\in \bS^{2N-1} \cap H$, $\abs{\det(\ud\varphi(u))} = u_1^{-2N}$,
  where~$\bS^{2N-1}$ and~$\bR^{2N-1}$ are considered with their usual Riemannian structures.
  \label{lem:jacproj}
\end{lemma}

\begin{proof}
  Let~$\psi : u \in \bR^{2N} \cap H \to  \left(
  \frac{u_2}{u_1}, \dotsc, \frac{u_{2N}}{u_1} \right)$, so that~$\varphi$ is
  the restriction of~$\psi$ to the sphere~$\bS^{2N-1}$.
  Firstly, the matrix of~$\ud\psi(x)$, for some~$x\in \bR^{2N}$, in the standard bases of~$\bR^{2N}$ and~$\bR^{2N-1}$,
  is given by
  \begin{equation} \mop{Mat}\left( \ud\varphi(x) \right) =
    \frac{1}{x_1^2}\begin{pmatrix}
      - x_2   & x_1   &           & 0       \\
        \vdots      &   & \ddots  &          \\
        - x_{2N} & 0       &          & x_1  
      \end{pmatrix}.
      \label{eqn:matdphi}
  \end{equation}
  Let~$u\in \bS^{2N-1} \cap H$.
  We may assume without loss of generality that~$u$ is of the form~$(u_1,u_2,0,\dotsc,0)$, with~$u_1^2+u_2^2=1$, because~$\abs{\det(\ud \varphi(u))}$
  is invariant under any unitary transformation of~$u$ that preserves the first coordinate.
  Let~$T \subset \bR^{2N}$ be the tangent space at~$u$ of~$\bS^{2N-1}$.
  Naturally, $\ud\varphi(u) = \ud\psi(u)_{|T}$.
  An orthonormal basis of~$T$ is given by~$\{ f, e_3, \dotsc, e_{2N} \}$,
  where~$f = (-u_2,u_1,0,\dotsc,0)$ and where~$e_i$ is the~$i$th coordinate vector.
  Using Equation~\eqref{eqn:matdphi}, we compute that~$\ud\varphi(u)(f) = u_1^{-2} e_1$ and that~$\ud\varphi(u)(e_i) = {u_1}^{-1} e_{i-1}$, for~$3\leq i \leq 2N$.
  Thus~$\abs{\det\left( \ud\varphi(u) \right)} = u_1^{-2N}$.
\end{proof}

The orthogonal monomial basis of~$\cH$ gives an identification~$\cH\simeq \bR^{2N}$
and we define this way the truncation~$\flo{f}_Q$ and the fractional part~$\left\{ f \right\}_Q$
of a polynomial system~$f\in\SH$.
The derandomization relies on finding a approximate root of~$\flo{f}_Q$, for
some~$Q$ large enough, and using~$\left\{ f \right\}_Q$ as the source of
randomness for the Beltrán-Pardo procedure.  Namely, we
compute~$\mop{HC}(\flo{f}_Q, \mop{BP}(\left\{ f \right\}_Q))$.
Almost surely, this computation produces an
approximate root of~$\flo{f}_Q$.  If~$Q$ is large enough, it is also an approximate
root of~$f$.  The main technical difficulty is to choose a precision and to ensure
that the result is correct while keeping the complexity under control.

\subsection{Homotopy continuation with precision check}

Let~$f$, $f'$, $g\in\bS(\polsys)$ and let~$\eta\in\bP^n$ be a root of~$g$.
Throughout this section, we assume that $d_\bS(f,f')\leq \rho$, for some~$\rho
> 0$ %
and that~$d_\bS(f,g)\leq \pi/2$. Up to
changing~$g$ into~$-g$, the latter is always true, since $d_\bS(f,-g)=\pi -
d_\bS(f,g)$. The notations~$I_2$, $M$ and~$\tilde M$ used in this section have been introduced in~\S\ref{sec:hc-algo}.
If~$\rho$ is small enough, then~$\mop{HC}(f',g,\eta)$ is an approximate root
not only of~$f'$ but also of~$f$.  But if~$\rho$ fails to be small enough,
$\mop{HC}(f',g,\eta)$ may not even terminate or, to say the least, 
$\mop{HC}(f',g,\eta)$ may take arbitrarily long to compute something that is
not an approximate root of~$f$.  To control the complexity of the new algorithm, it
is important to be able to recognize this situation at least as fast
as~$\mop{HC}(f,g,\eta)$ would terminate.

As in~\S\ref{sec:hc-algo}, let~$f_t = \Gamma(g,f,t)$ and~$f'_t = \Gamma(g, f',
t)$.  Let~$t\in J \to \zeta_t\in\bP^n$ be the homotopy continuation associated
to~$f_t$, on~$[0,1]$, and~$t\in J' \to \zeta'_t\in\bP^n$ be the one associated
to~$f'_t$, defined on some maximal intervals~$J, J'\subset [0,1]$
containing~$0$.  Let~$\mu_t = \mu(f_t,\zeta_t)$ and~$\mu'_t =
\mu(f'_t,\zeta'_t)$.

\begin{lemma} \label{lem:Gcont}
  $d_{\bS}(f_t,f'_t) \leq 2 d_{\bS}(f,f')$ for any~$t\in [0,1]$.
\end{lemma}

\begin{proof}
  Let~$\alpha_t = d_{\bS}(f_t,f'_t)$, $\beta = d_\bS(f,g) \leq \frac{\pi}{2}$
  and~$\gamma = d_\bS(f',g)$.  Without loss of generality, we may assume
  that~$\alpha_1 <
  \frac\pi2$, otherwise the inequality~$\alpha_t \leq 2\alpha_1$ that we want
  to check is trivial.
  
  The spherical law of cosines applied to the spherical triangle~$\left\{ g,f_t,f'_t\right\}$ gives the equality
  \begin{align}\label{eqn:cosat}
    \cos \alpha_t
      &= \cos(t\beta)\cos(t\gamma) + \sin(t\beta)\sin(t\gamma) \cos A,%
  \end{align}
  where~$A$ is the angle of the triangle at~$g$.
  We deal with three cases. Firstly, we assume that~$\gamma \leq \frac{\pi}{2}$.
  Then~$\cos \alpha_t$ decreases at~$t$ increases: Indeed, Equation~\eqref{eqn:cosat} rewrites as
  \begin{equation}\label{eqn:cosatbis}
    \cos\alpha_t = \cos(t\beta - t\gamma) - \sin(t\beta)\sin(t\gamma)(1-\cos A)
  \end{equation}
  and, as~$t$ increases, $\cos(t\beta - t\gamma)$ decreases, because~$|\beta-\gamma|\leq \pi$, and both~$\sin(t\beta)$ and~$\sin(t\gamma)$ increase, because~$\beta,\gamma \leq \frac\pi2$.
  Thus~$\cos\alpha_t \geq \cos\alpha_1$, for~$0\leq t\leq 1$, and it follows that~$\alpha_t \leq \alpha_1$.
  
  Second case, we assume that~$\gamma > \frac\pi2$ and~$\beta = \frac\pi2$.
  For~$t\in[0,1]$, Equation~\eqref{eqn:cosatbis} shows that
  \[ \cos\alpha_t \geq \cos(\tfrac\pi2-\gamma) - (1-\cos A) = \sin\gamma + \cos A - 1,  \]
  using~$\cos(t\beta - t\gamma) \geq \cos(\beta-\gamma)$ and~$1-\cos(A) \geq 0$.
  Equation~\eqref{eqn:cosat} shows that~$\cos\alpha_1 = \sin \gamma \cos A$.
  In particular~$\cos A \geq 0$, since~$\alpha_1 \leq \frac\pi2$ and~$\sin\gamma \geq 0$.
  It follows that
  \[2 \sin^2\gamma \cos^2 A \leq {\sin^4 \gamma + \cos^4 A} \leq \sin \gamma + \cos A, \]
  and finally that~$\cos(2\alpha_1) \leq \cos\alpha_t$, because~$\cos(2\alpha_1) = 2\cos^2\alpha_1 -1$.
  Since~$2\alpha_1 \leq \pi$, we obtain that~$2\alpha_1 \geq \alpha_t$, which concludes in the second case.

  Third case, we assume only that~$\gamma > \frac\pi2$ (and always~$\beta \leq \frac\pi2$).
  Let~$h\in\SH$ be the unique point on the spherical segment~$[f,f']$ such that~$d_\bS(g,h) = \frac{\pi}{2}$.
  In particular, we have that~$d_\bS(f,f') = d_\bS(f,h) + d_\bS(h,f')$ and
  \[ \alpha_t = d_{\bS}(f_t,f'_t) \leq  d_\bS(f_t,h_t) + d_\bS(h_t,f'_t), \]
  where~$h_t \eqdef \Gamma(g, h, t)$.
  The first case shows that~$d_\bS(f_t,h_t) \leq d_\bS(f,h)$
  and the second case shows that~$d_\bS(h_t,f'_t) \leq 2 d_\bS(h, f')$.
  Thus~$\alpha_t \leq  d_\bS(f,h) + 2 d_\bS(h, f') \leq 2 d_\bS(f,f')$.
\end{proof}

Recall that~$M(f,g,\zeta)$ denotes~$\sup_{t\in J} \mu_t$, see~\S\ref{sec:hc-algo} and Equation~\eqref{eqn:def-M-Ip}.

\begin{lemma} \label{lem:maxmuhc}
  If~$D^{3/2}M(f,g,\zeta)^2\rho \leq \frac{1}{168}$ then~$J=J'=[0,1]$ and for any~$t\in [0,1]$:
  \begin{enumerate}[(i)]
    \item $(1+\varepsilon)^{-1} \mu'_t \leq \mu_t \leq (1+\varepsilon) \mu'_t$;
    \item $D^{3/2} \mu_t d_\bP(\zeta_t,\zeta'_t) \leq \frac{1}{51}$.
  \end{enumerate}
\end{lemma}

\begin{proof}
  The assumption implies that~$M(f,g,\zeta) < \infty$, and thus~$J=[0,1]$, by Lemma~\ref{lem:maxhomcont}.
  Let~$S$ the set of all~$t\in J'$ such that~$D^{3/2}\mu_t d_\bP(\zeta_t,\zeta'_t) \leq \frac{1}{51}$.
  It is a nonempty closed subset of~$J'$.
  Let~$t\in S$. By Lemma~\ref{lem:Gcont}, we have~$d_\bP(f_t,f'_t) \leq 2\rho$,
  so
  \[ D^{3/2} \mu^2_t d_\bP(f_t,f'_t) \leq \frac{2}{112} = \frac{\varepsilon}{4(1+\varepsilon)}. \]
  Proposition~\ref{prop:condroot} implies that there exists a root~$\eta$ of~$f'_t$ such that~$d_\bP(\eta,\zeta_t)\leq 2 (1+\varepsilon) \mu_t \rho$
  and $(1+\varepsilon)^{-1} \mu_t \leq \mu(f'_t,\eta)\leq (1+\varepsilon) \mu_t$.
  Because~$d_\bP(\eta,\zeta'_t) \leq d_\bP(\eta,\zeta_t) + d_\bP(\zeta_t,\zeta'_t)$ and~$t\in S$ we obtain
  \[ D^{3/2} \mu(f'_t,\eta) d_\bP(\eta,\zeta'_t) \leq D^{3/2}(1+\varepsilon) \mu_t \left( 2 (1+\varepsilon) \mu_t \rho + \frac{1}{51 D^{3/2} \mu_t }\right) \leq (1+\varepsilon)^2 \frac{2}{112} + (1+\varepsilon)\frac{1}{51}  \leqcalc \frac{1}{3}. \]
  and Theorem~\ref{thm:gthm} implies that~$\zeta'_t$ approximates~$\eta$ as a root of~$f'_t$. Since it is also an exact root of~$f'_t$, this implies~$\zeta'_t=\eta$.
  In particular~$D^{3/2} \mu_t d_\bP(\zeta'_t,\zeta_t)\leq 2 (1+\varepsilon) D^{3/2} \mu^2_t \rho \ltcalc \frac{1}{51}$.
  Thus~$t$ is in the interior of~$S$, which proves that~$S$ is open and finally that~$S=J$.
  Moreover, since~$\mu'_t\leq (1+\varepsilon)\mu_t$, $\mu'_t$ is bounded on~$J'$, thus~$J'=[0,1]$.
\end{proof}

\begin{algo}[tp]
  \centering
  \begin{algorithmic}
    \Function{$\mop{HC}'$}{$f$, $g$, $z$, $\rho$}
      \State $t\gets {1}/\left(101 D^{3/2} \mu(g, z)^{2} d_\bS(f,g)\right)$
      \State $h\gets g$
      \While{$1 > t$ and $D^{3/2}\mu(h,z)^2 \rho \leq \frac{1}{151}$}
        \State $h \gets \Gamma(g, f, t)$
        \State $z \gets \cN(h, z)$
        \State $t \gets t + {1}/\left(101 D^{3/2} \mu(h, z)^{2} d_\bS(f,g)\right)$        
      \EndWhile
      \If{$D^{3/2}\mu(h,z)^2 \rho > \frac{1}{151}$} %
          \textbf{return} \textsc{fail}
       \Else{} 
          \textbf{return} $z$
      \EndIf
    \EndFunction
  \end{algorithmic}
  \caption[]{Homotopy continuation with precision check
    \begin{description}
      \item[Input.] $f$, $g \in \bS(\cH)$, $z\in\bP^n$ and~$\rho > 0$.
      \item[Output.] $w\in\bP^n$ or \textsc{fail}.
      \item[Specifications.] See Proposition~\ref{prop:hcp}. 
    \end{description}
  }\label{algo:hcpc}
\end{algo}

This leads to the procedure~$\mop{HC}'$, see Algorithm~\ref{algo:hcpc}.  It
modifies procedure~$\mop{HC}$ (Algorithm~\ref{algo:hc}) in only one respect:
each iteration checks up on the failure condition~$D^{3/2}\mu(h,z)^2 \rho >
\frac{1}{151}$.  If the failure condition is never met, then~$\mop{HC}'$
computes exactly the same thing as~$\mop{HC}$.
Recall that~$\tilde M(f',g,\eta)$ denotes the maximum condition number~$\mu$
that arises in the homotopy continuation~$\mop{HC}(f',g,\eta)$,
and that~$I_p(f,g,\eta)$ denote the integral of~$\mu^p$ along the homotopy path from~$g$ to~$f$, see~\S\ref{sec:hc-algo} and Equation~\eqref{eqn:def-M-Ip}.

\begin{proposition}
  \label{prop:hcp}
  If~$d_\bS(f,g)\leq \frac\pi2$ and~$d(f,f')\leq \rho$, then
  the procedure~$\mop{HC}'(f',g,\eta,\rho)$:
  \begin{enumerate}[(i)]
    \item\label{it:hcp:comp} terminates and performs at most $158\,D^{3/2} d_\bS(f,g) I_2(f,g,\eta)+4$ steps;
    \item\label{it:hcp:approx} outputs an approximate root of~$f$, or fails;
    \item\label{it:hcp:succond} succeeds (\emph{i.e.} outputs some~$z\in\bP^n$) if and only if~$D^{3/2} \tilde M(f',g,\eta)^2 \rho \leq  \frac{1}{151}$;
    \item\label{it:hcp:succond2} succeeds if~$D^{3/2}M(f,g,\eta)^2 \rho \leq \frac{1}{236}$.
  \end{enumerate}
\end{proposition}

\begin{proof}
  At each iteration, the value of~$t$ increases by at least~${151 \rho}/(101
  d_\bS(f',g))$, thus there are at most~${101 d_\bS(f',g)}/(151 \rho)$
  iterations before termination.

  By construction, the procedure~$\mop{HC}'(f',g,\eta,\rho)$ fails if and only
  if at some point of the procedure~$\mop{HC}(f',g,\eta,\rho)$ it happens
  that~$D^{3/2}\mu(h,z)^2 \rho > \frac{1}{151}$.  In other words, the
  procedure~$\mop{HC}'(f',g,\eta,\rho)$ fails if and only if~$D^{3/2} \tilde
  M(f',g,\eta)^2 \rho > \frac{1}{151}$, by definition of~$\tilde M$.  And since
  the procedure terminates, it succeeds if and only if it does not fail.  This
  proves~\ref{it:hcp:succond}.

  Let us bound the number~$K'(f',g,\eta,\rho)$ of iterations of the
  procedure~$\mop{HC}'(f',g,\eta,\rho)$ before termination.
  If~$\mop{HC}'(f',g,\eta,\rho)$ succeeds, then~$K'(f',g,\eta,\rho) =
  K(f',g,\eta)$. Furthermore
  \begin{equation}
    K'(f',g,\eta,\rho) = \sup \left\{ K(f'_s,g,\eta) \st s\in[0,1], \ \mop{HC}'(f'_s,g,\eta,\rho) \text{ succeeds} \right\}.
    \label{eqn:proof:hcp:Kp}
  \end{equation}
  Let~$s\in [0,1]$ such that~$\mop{HC}'(f'_s,g,\eta,\rho)$ succeeds, that is to
  say~$D^{3/2} \tilde M(f'_s,g,\eta)^2 \rho \leq \frac{1}{151}$.
  Theorem~\ref{thm:complexity-hc}\ref{it:Mtilde} shows that
  \[ (1+\varepsilon)^{-2} M(f'_s,g,\eta) \leq \tilde M(f'_s,g,\eta) \leq (1+\varepsilon)^2 M(f'_s,g,\eta). \]
  In particular~$D^{3/2}M(f'_s,g,\eta)^2\rho \leq \frac{(1+\varepsilon)^4}{151} \leq \frac{1}{112}$ and
  Lemma~\ref{lem:maxmuhc} shows that~$(1+\varepsilon)^{-2} \leq \mu'_t\leq (1+\varepsilon)\mu_t$ for
  all~$t\leq s$. So we obtain that~$(1+\varepsilon)^{-2} I_2(f_s,g,\eta) \leq I_2(f'_s,g,\eta)\leq (1+\varepsilon)^2 I_2(f_s,g,\eta)$ and
  \begin{align*}
    K(f'_s,g,\eta) &\leq 136\, D^{3/2} d_\bS(f'_s,g) I_2(f'_s,g,\eta) && \text{by Theorem~\ref{thm:complexity-hc}\ref{it:boundK}} \\
    &\leq 136(1+\varepsilon)^2 D^{3/2} \left( d_\bS(f_s,g) + 2\rho \right) I_2(f_s,g,\eta) && \text{by Lemma~\ref{lem:Gcont}.} 
  \end{align*}
  Besides
  $D^{3/2} I_2(f_s,g,\eta) \rho \leq (1+\varepsilon)^2 D^{3/2}  M(f'_s,g,\eta)^2 \rho \leq \frac{(1+\varepsilon)^2}{112}$,
  so we obtain
    \begin{align*}
    K(f'_s,g,\eta) \leqcalc 158 D^{3/2} d_\bS(f_s,g) I_2(f_s,g,\eta) + 4 \leq  158 D^{3/2} d_\bS(f,g) I_2(f,g,\eta) + 4.
  \end{align*}
  Together with Equation~\eqref{eqn:proof:hcp:Kp}, this completes the proof of~\ref{it:hcp:comp}.
      
  Let us assume that the procedure~$\mop{HC}'(f',g,\eta,\rho)$ succeeds and
  let~$z$ be its output, which is nothing but~$\mop{HC}(f',g,\eta)$.
  Theorem~\ref{thm:complexity-hc}\ref{it:HCapproxb} shows that~$D^{3/2}\mu'_1
  d_\bP(z,\zeta'_1) \leq \frac{1}{23}$, where~$\zeta'_1$ is the root of~$f'_1 = f'$ obtained by homotopy continuation. 
  As above, with~$s=1$, we check that~$\mu_1 \leq (1+\varepsilon)\mu'_1$ and~$D^{3/2} \mu'_1
  d_{\bP}(\zeta_1,\zeta'_1) \leq \frac{1}{51}$ using Lemma~\ref{lem:maxmuhc}.
  Thus
  \[ D^{3/2} \mu_1 d_\bP(z,\zeta_1) \leq (1+\varepsilon)\left( \frac{1}{23} + \frac{1}{51} \right) \ltcalc \frac{1}{3}. \]
  Then~$z$ approximates~$\zeta_1$ as a root of~$f_1$, by Theorem~\ref{thm:gthm}. This proves~\ref{it:hcp:approx}.

  Lastly, let us assume that~$D^{3/2}M(f,g,\eta)^2 \rho \leq \frac{1}{236}$.
  Lemma~\ref{lem:maxmuhc} implies that $M(f,g,\eta) \geq (1+\varepsilon)^{-1} M(f',g,\eta)$
  and Theorem~\ref{thm:complexity-hc}\ref{it:Mtilde} shows that~$M(f',g,\eta) \leq (1+\varepsilon)^{2} \tilde M(f',g,\eta)$.
  Thus
  \[ D^{3/2} \tilde M(f',g,\eta)^2 \rho \leq (1+\varepsilon)^6 D^{3/2} M(f,g,\eta)^2 \rho \leq \frac{(1+\varepsilon)^6}{236} \leqcalc \frac{1}{151} \]
  and~$\mop{HC}'(f',g,\eta,\rho)$ succeeds. This proves~\ref{it:hcp:succond2}.
\end{proof}

\subsection{A deterministic algorithm}

Let~$f\in \SH$ be the input system to be solved and let~$Q \geq 1$ be a given
precision.  We compute
\[ f'=\flo{f}_Q,\  (g,\eta) = \mop{BP}(\left\{ f \right\}_Q), \  \varepsilon = \mop{sign}(\pi/2-d_\bS(f,g))
  \ \text{and}\  \rho = {3 N^{1/2}} / Q. \]
Lemma~\ref{lem:qindepA} shows that~$d_\bS(f,f')\leq \rho$.
Then we run the homotopy continuation procedure with
precision check $\mop{HC}'( f', \varepsilon g, \eta,\rho)$, which may
fail or output a point~$z\in\bP^n$.  If it does succeed, then
Proposition~\ref{prop:hcp} ensures that~$z$ is an approximate root of~$f$.  If
the homotopy continuation fails, then we replace~$Q$ by~$Q^2$ and we start
again, until the call to~$\mop{HC}'$ succeeds. This leads to the deterministic
procedure~$\mop{DBP}$, Algorithm~\ref{algo:dbp}.  If the computation
of~$\mop{DBP}(f)$ terminates then the result is an approximate root of~$f$.
Section~\ref{sec:avanalysis} studies the average number of homotopy steps
performed by~$\mop{DBP}(f)$ while Section~\ref{sec:bss} studies the average
total cost of an implementation of~$\mop{DBP}$ in the BSS model extended with the square root.

\begin{algo}[tp]
  \centering
  \begin{algorithmic}
    \Function{DBP}{$f$}
      \State $Q \gets N$
      \Repeat
        \State $Q \gets Q^2$
        \State $f' \gets \flo{f}_Q$
        \State $(g,\eta) \gets \mop{BP}(\left\{ f \right\}_Q)$
        \State $\varepsilon \gets \mop{sign}(\Re\langle f, g\rangle)$
        \State $\rho \gets {(2N)^{1/2}} / Q$
        \State $z \gets \mop{HC}'(f', \varepsilon g, \eta,\rho)$
      \Until{$\mop{HC}'$ succeeds}
      \State\textbf{return} $z$
    \EndFunction
  \end{algorithmic}
  \caption[]{Deterministic variant of Beltrán-Pardo algorithm
\begin{description}
  \item[Input.] $f\in\polsys$
  \item[Output.] $z\in\bP^n$
  \item[Postcondition.] $z$ is an approximate root of~$f$
\end{description}
  }
  \label{algo:dbp}
\end{algo}

\subsection{Average analysis}
\label{sec:avanalysis}

Let~$f\in \SH$ be the input system, a \uuv  , and we
consider a run of the procedure~$\mop{DBP}(f)$.  Let~$Q_k$ be the precision at
the~$k${th} iteration, namely~$Q_k = N^{2^k}$.\marginpar{$Q_k$}  We set also\marginpar{$f_k$, $g_k$, $\eta_k$, $\varepsilon_k$, $\rho_k$}
\[ f_k=\flo{f}_{Q_k},\  (g_k,\eta_k) = \mop{BP}(\left\{ f \right\}_{Q_k}), \  \varepsilon_k = \mop{sign}(\pi/2-d_\bS(f,g_k))
  \ \text{and}\  \rho_k = 3 N^{1/2}/Q_k. \]
Let~$\Omega$\marginpar{$\Omega$} be the least~$k$ such that the homotopy
continuation with precision check $\mop{HC}'(f_k, \varepsilon_k g_k, \eta_k,
\rho_k)$ succeeds.  Note that~$\Omega$ is a random variable.  To perform the
average analysis of the total number of homotopy steps, we first deal with each
iteration separately (Lemmas~\ref{lem:qindep} and~\ref{lem:isotrop}) and then
give tail bounds on the probability distribution of~$\Omega$
(Proposition~\ref{lem:omega}). Even if the number of steps in each iteration
are not independent from each other and from~$\Omega$, Hölder's inequality
allows obtaining a bound on the total number of steps (Theorem~\ref{thm:totsteps}).

Let~$(g,\eta)\in V$\marginpar{$g$, $\eta$} be a random variable with distribution~$\rstd$ and independent of~$f$.

\begin{lemma} \label{lem:qindep}
  Let~$\Theta : \cH\times V \to\bR$ be any nonnegative measurable function.
  For any~$k\geq 1$,
  \[ \bE\left( \Theta(f_k, \varepsilon_k g_k, \eta_k) \right)  \leq  10 \bE\left( \Theta(f_k, g, \eta) \right). \]
\end{lemma}
\begin{proof}
  It is an application of Lemma~\ref{lem:qindepA}.
  We first remark that~$\varepsilon_k \in \left\{ -1,1 \right\}$ so
  \[ \Theta\left(f_k, \varepsilon_k g_k, \eta_k \right) \leq \Theta\left(f_k, g_k, \eta_k \right) + \Theta\left(f_k, -g_k, \eta_k \right). \]
  Then
  \begin{align*}
    \bE\left( \Theta(f_k, g_k, \eta_k) \right) &=
    \frac{1}{\vol(\SH)}\int_\SH \Theta\left( \flo{f}_{Q_k}, \mop{BP}(\left\{ f \right\}_{Q_k}) \right) \ud f \\
    &\leq \frac{\exp\left( \frac{2 N^{3/2}}{Q_k} \right)}{\vol(\SH)^2} \int_{\SH\times\SH} \Theta\left( \flo{f}_{Q_k}, \mop{BP}(g)\right)\ud f\ud g 
    && \text{by Lemma~\ref{lem:qindepA}}\\
    &= \frac{\exp\left( \frac{2 N^{3/2}}{Q_k} \right)}{\vol(\SH)} \int_{\cH}\int_{V} \Theta\left( \flo{f}_{Q_k}, g, \eta \right) \ud f \ud \rstd(g,\eta) && \text{by Theorem~\ref{thm:bprand}}\\
    &= \exp\left( \tfrac{2 N^{3/2}}{Q_k} \right) \bE\left( \Theta(f_k, g, \eta) \right).
  \end{align*}
  Similarly, $\bE\left( \Theta(f_k, -g_k, \eta_k) \right) \leq \exp\left( \tfrac{2 N^{3/2}}{Q_k} \right) \bE\left( \Theta(f_k, -g, \eta) \right)$,
  and since~$g$ and~$-g$ have the same probability distribution, $\bE\left( \Theta(f_k, -g, \eta) \right) =  \bE\left( \Theta(f_k, g, \eta) \right)$.
  To conclude, we remark that~$Q_k\geq N^{2}$ and that~$e^{\sqrt{2}}\leq 5$.
\end{proof}

\begin{lemma}
  $\bE(I_p(f, g,\eta)) = \bE(\mu(g,\eta)^{p})$ for any~$p \geq 1$ and~$k\geq 1$.
  \label{lem:isotrop}
\end{lemma}
\begin{proof}

  Let~$h_t = \Gamma(g,f,t)$, for~$t\in[0,1]$, and let~$\zeta_t$ be the associated
  homotopy continuation.  Let~$\tau\in[0,1]$ be a \uuv independent from~$f$ and~$(g,\eta)$.
  Clearly~$\bE(I_{p}(f,g,\eta)) = \bE(\mu(h_\tau, \zeta_\tau)^{p})$, 
  so it is enough to prove that~$(h_\tau,\zeta_\tau) \sim \rstd$.
  The systems~$f$ and~$g$ are independent and uniformly distributed on~$\bS(\polsys)$.
  So their probability distributions is invariant under any unitary transformation of~$\polsys$.
  Then so is the probability distribution of~$h_t$ for any~$t\in [0,1]$,
  and there is a unique such probability distribution: the uniform distribution on~$\SH$.
  The homotopy continuation makes a bijection between the roots of~$g$ and those of~$h_t$.
  Since~$\eta$ is uniformly chosen among the roots of~$g$, so is~$\zeta_t$ among the roots of~$h_t$.
  That is, $(h_t,\zeta_t) \sim \rstd$ for all~$t\in[0,1]$, and then~$(h_\tau,\zeta_\tau) \sim \rstd$.
\end{proof}

\begin{proposition}
  $\bP(\Omega > k) \leq 2^{17}\, D^{9/4} n^{3/2} N^{7/4} Q_k^{-1/2}$. 
  \label{lem:omega}
\end{proposition}

\begin{proof}
  The probability that~$\Omega > k$ is no more than the probability that $\mop{HC}'(f_k, g_k, \eta_k, \rho_k)$ fails.
  By Lemma~\ref{lem:qindep}, $\bP\left( \mop{HC}'(f_k, \varepsilon_k g_k, \eta_k, \rho_k)\text{ fails} \right) \leq 10\, \bP\left( \mop{HC}'(f_k, g, \eta, \rho_k)\text{ fails} \right)$.
  Given that~$d_\bS(f,f_k)\leq \rho_k$,
  \begin{align*}
   \bP\left( \mop{HC}'(f_k,  g, \eta, \rho_k)\text{ fails} \right)
   &\leq \bP\left( D^{3/2}M(f, g,\eta)^2 \rho_k \geq \frac{1}{236} \right) && \text{by Proposition~\ref{prop:hcp}\ref{it:hcp:succond2}} \\
    &\leq \bP\left( D^{9/2} I_3(f, g,\eta)^2 \rho_k \geq \frac{1}{236\cdot 151^2} \right) && \text{by Proposition~\ref{lem:maxmu}}\\
    &\leq 151 \cdot 236^{1/2}\, D^{9/4} \rho_k^{1/2} \, \bE\left( I_3(f, g,\eta) \right) && \text{by Markov's inequality.}
  \end{align*}
  Lemma~\ref{lem:isotrop} and Theorem~\ref{thm:mumoment} imply then
  \[ \bE\left(I_3(f, g,\eta) \right) \leq \bE\left( \mu(g,\eta)^3 \right) \leq 3 (nN)^{3/2}. \]
  All in all, and since~$\rho_k = 3N^{1/2}/Q_k$,
  \[ \bP(\Omega > k) \leq 10\cdot (151 \cdot 236^{1/2}\, D^{9/4}) \cdot (3N^{1/2}/Q_k)^{1/2} \cdot 3 (nN)^{3/2} \leqcalc 2^{17}\, D^{9/4} n^{3/2} N^{7/4} Q_k^{-1/2} \qedhere\] 
\end{proof}

Let~$K(f)$\marginpar{$K(f)$} be the total number of homotopy steps performed by
procedure~$\mop{DBP}(f)$ and let the number of homotopy steps performed by
procedure $\mop{HC}'(f_k, \varepsilon_k g_k, \eta_k, \rho_k)$ be denoted by
$K'(f_k, \varepsilon_k g_k, \eta_k, \rho_k)$, so that
\[ K(f) = \sum_{k=1}^\Omega K'(f_k, \varepsilon_k g_k, \eta_k, \rho_k), \]

\begin{theorem}
  If~$N\geq 21$ then~$\bE(K(f)) \leq 2^{17}\, n D^{3/2} N$.
  \label{thm:totsteps}
\end{theorem}

\begin{proof}
  Let~$X_k = K'(f_k, \varepsilon_k g_k, \eta_k, \rho_k)$ and let~$0< p \leq \frac32$.
  By Lemma~\ref{lem:qindep} and Proposition~\ref{prop:hcp}\ref{it:hcp:comp},
  \begin{align*}
    \bE(X_k^p)^{1/p}
      &\leq 10\, \bE\left( \left( 158\, D^{3/2} d_\bS(f, g) I_2(f, g,\eta) + 4\right)^{p} \right)^{1/p}, \\
      \intertext{and because~$d_\bS(f, g)\leq {\pi}$ and by Minkowski's inequality, we obtain}
      &\leq   10 \left( 158\, D^{3/2}{\pi} \bE\left( I_{2}(f, g,\eta)^p  \right)^{1/p} +  4 \right).
  \end{align*}
  Jensen's inequality implies that
  $I_2(f,g,\eta)^p \leq  I_{2p}(f, g,\eta)$.
  Then~$\bE\left( I_{2p}(f, g,\eta)  \right) \leq \frac{3}{4-2p}(nN)^p \leq 3(nN)^p$, by Lemma~\ref{lem:isotrop} and Theorem~\ref{thm:mumoment}.
  In the end,
  \begin{equation}
    \bE(X_k^p)^{1/p} \leqcalc 15000\, n D^{3/2} N.
    \label{eqn:pmom}
  \end{equation}
  Now, let~$p =\log N/(\log N -1)$.  If~$N \geq 21$ then~$p \leqcalc \frac32$.
  We write the expectation of~$K(f)$ as
  \[ \bE(K(f)) = \bE\big( \sum_{k=1}^\Omega X_k \big) = \sum_{k=1}^\infty \bE(X_k \mbbone_{\Omega\geq k}). \] 
  Let~$q=1/\log N$, so that~$\frac1p +\frac 1q =1$. From Hölder's inequality,
  $\bE(X_k \mbbone_{\Omega\geq k}) \leq \bE(X_k^p)^{1/p} \bP(\Omega \geq k)^{1/q}$
  and thus
  \[ \bE(K(f)) \leq \max_{k\geq 1} \bE(X_k^p)^{1/p} \sum_{k=1}^\infty\bP(\Omega \geq k)^{1/q}. \]
  Lemma~\ref{lem:sumW} below, with~$C=1$, $L=4$ and~$\delta = 1/q$, shows that
  \[ \sum_{k=1}^\infty\bP(\Omega \geq k)^{1/q} \leq L+1 + \frac{2^{17/\log 21} e^5}{e^{2^L} - 1} \leqcalc 6. \]
  The claim follows then from Equation~\eqref{eqn:pmom} and~$6\cdot 15000 \leqcalc 2^{17}$.
\end{proof}

\begin{lemma}\label{lem:sumW}
  For any~$C, \delta > 0$ and any integer~$L \geq 2$ such that~$C < N^{\delta 2^L}$,
  \[ \sum_{k= 1}^{\infty} C^k \bP(\Omega \geq k)^{\delta} \leq \sum_{k=1}^{L+1} C^k + \frac{\left( 2^{17} N^5 \right)^{\delta} C^{L+2}}{N^{\delta 2^L} - C}. \]
\end{lemma}

\begin{proof}
  For any~$k$, Proposition~\ref{lem:omega} implies that
  \[  \bP(\Omega \geq k) = \bP(\Omega > k-1) \leq 2^{17}\, D^{9/4} n^{3/2} N^{7/4} Q_{k-1}^{-1/2} \leq  2^{17}\, N^5 N^{-2^{k-2}}, \]
  using~$D\leq N$ and~$n^2\leq N$.
  Moreover, $2^{p-1} \geq p$, for any integer~$p$, so that~$N^{-2^{k-2}} \leq N^{- 2^{L} (k-L-1)}$.
  Of course, it also holds that~$\bP(\Omega \geq k) \leq 1$.
  Thus
  \begin{align*}
    \sum_{k\geq 1}^{\infty} C^k \bP(\Omega \geq k)^{\delta} &\leq \sum_{k=1}^{L+1} C^k + \left( 2^{17} N^5  \right)^\delta \sum_{k = L+2}^{\infty} C^k N^{- \delta 2^{L} (k-L-1)},
  \end{align*}
  and the latter sum is a geometric sum which evaluates to~$C^{L+2}/(N^{\delta 2^L} - C)$.
\end{proof}

\subsection{Implementation in the BSS model with square root}
\label{sec:bss}

Algorithms~$\HCp$ and~DBP (Algorithms~\ref{algo:hcpc} and~\ref{algo:dbp}  respectively) have been described assuming the
possibility to compute exactly certain nonrational functions: the square root,
the trigonometric functions sine and cosine and the operator norm of a linear
map. A BSS machine can only approximate them, but it can do it efficiently. I
propose here an implementation in the BSS model extended with the ability 
of computing the square root of a positive real number at unit cost.  We could
reduce further to the plain BSS model at the cost of some lengthy and nearly irrelevant
technical argumentation. 
We now prove the main result of this article:

\begin{theorem}
  \label{thm:main}
  There exists a BSS machine~$A$ with square root and a constant~$c > 0$
  such that for any positive integer~$n$ and any positive integers~$d_1$, \dots, $d_n$:
  \begin{enumerate}[(i)]
    \item\label{it:correct} $A(f)$ computes an approximate root of~$f$ for almost all~$f\in \cH$;
    \item\label{it:totcomp} if~$f\in\SH$ is a \uuv, then the average number of operations performed by~$A(f)$ is at most~$c n D^{3/2} N (N+n^3)$.
  \end{enumerate}
\end{theorem}

 Firstly, we describe an implementation of
Algorithms~$\HCp$ and~DBP in the extended BSS model.
The first difficulty is the condition number~$\mu(f,z)$:  it rests upon the
operator norm for the Euclidean distance which is not computable with rational
operations.  While there are  efficient numerical algorithms to compute such an
operator norm in practice, it is not so easy to give an algorithm that
approximates it in good complexity in the BSS model.\footnote{See for example \cite{ArmBelBur15} or \cite{ArmCuc15};
unfortunately the Gaussian distribution that they assume does not fit the situation here.
}
Fortunately, we can easily compute the operator norm of a
matrix~$M\in\bC^{n\times n}$ within a factor~$2$ as
follows:\footnote{I thank one of the referees for having communicated this method to me.}
we first compute a tridiagonalization~$T$ of the Hermitian matrix~${\overline M} {}^t M$
with~$\cO(n^3)$ operations, using Householder's reduction,
and then\autocite{Kah66}
\[ \frac{1}{\sqrt{3}} \|T\|_{1} \leq \|M\|^2 \leq \| T\|_{1}, \] 
where~$\|T\|_{1}$ is the operator $\ell_1$-norm of~$T$, that is the maximum $\ell_1$-norm of a column.
Therefore, up to a few modifications in the constants, we may assume that~$\mu(f,z)$ is computable in~$\cO(n^3)$ operations, given~$\ud f(z)$.

The second difficulty lies in the use of the trigonometric functions sine and cosine.  They
first appear in the definition of the geodesic path~$\Gamma$,
Equation~\eqref{eqn:defG}, which is used in Algorithm~\ref{algo:hcpc}.
In the case where~$d_\bS(f,g)\leq \pi/2$, it is good enough to replace~$\Gamma(g,f,\delta)$ by
\[ \frac{\delta f +(1-\delta)g}{\|\delta f +(1-\delta)g \|}. \]
This is classical and implies  modifications in the constants only.\autocite[See for example][\S17.1]{BurCuc13}
The trigonometric functions also appear in Sibuya's function~$S$, see Equation~\eqref{eqn:sibuya}.
This issue can be handled with power series approximations:

\begin{lemma} \label{lem:trig}
  There is a BSS machine with square root that computes, for any~$N$ and any~$x\in[0,1]^{2N-1}$,
  with~$\cO(N\log N)$ operations,
  a point~$\tilde S(x) \in \bS^{2N-1}$ such that
  \[ \int_{[0,1]^{2N-1}} \Theta(\tilde S(x))\ud x \leq \frac{2}{\vol(\bS^{2N-1})}\int_{\bS^{2N-1}} \Theta(y) \ud y, \]
  for any nonnegative measurable function~$\Theta : \bS^{2N-1} \to \bR$.
\end{lemma}

\begin{proof}[Sketch of the proof]
  For any positive integer~$Q$, let~$F_Q(x)$ be the Taylor series expansion at~$0$,
  truncated at~$x^Q$, of the entire function~$(\exp(2i\pi x)-1)/(x-1)$.  It is a
  polynomial of degree~$Q$ that can be computed with~$\cO(Q)$ operations, assuming
  that~$\pi$ is a constant of the machine, by using the linear
  recurrence
  $(n+2)u_{n+2} = (2i\pi+n+2) u_{n+1}+2i\pi u_n$
  satisfied by the coefficients of~$F_Q$.
  Let~$\Cos_Q(x)$
  and~$\Sin_Q(x)$ be the real and imaginary parts of~$(1+(x-1)F_Q(x))/|(1+(x-1)F_Q(x))|$
  respectively.
  
  The function~$x\in[0,1]\to(\Cos_Q(x),\Sin_Q(x))$ gives a parametrization
  of the circle~$\bS^1$ whose Jacobian is almost constant:
  we can check that there is a universal constant~$C>0$ such that
  \[ \abs{\Cos_Q'(x)^2 + \Sin_Q'(x)^2 - (2\pi)^2} \leq C e^{-Q}. \]
  Thus for any continuous function~$\theta : \bS^1 \to \bR$
  \[ \int_0^1 \theta(\Cos_Q(x),\Sin_Q(x))\ud x \leq \frac{1+C e^{-Q}}{2\pi} \int_{\bS^1} \theta(y) \ud y. \]
  Let~$\tilde S$ be the function~$[0,1]^{2N-1}\to \bS^{2N-1}$ defined in the same way as~$S$, Equation~\eqref{eqn:sibuya},
  but with~$\Cos_Q$ and~$\Sin_Q$ in place of~$\sin$ and~$\cos$ respectively, with some~$Q \sim \log N$
  such that~$(1+C e^{-Q})^N \leq 2$.
  It is easy to check that~$\tilde S$ satisfies the desired properties.
\end{proof}

In Algorithm~DBP, there is no harm in using~$\tilde S$ in place
of~$S$.  We obtain this way variants of Algorithms~$\HCp$ and~DBP
that fit in the BSS model with square root.
It remains to evaluate the overall number of operations.
It is well known that~$f(z)$ and~$\ud f(z)$ can be computed at a point~$z\in \bC^{n+1}$ in~$\cO(N)$ operations---the latter
as a consequence of a theorem of Baur and
Strassen\autocite{BauStr83}.
Together with the approximate computation of the operator norm discussed above, this implies the following:

\begin{lemma}
  \label{lem:costnewt}
  There exists a BSS machine with square root that compute~$\mu(f,z)$ (within a factor~$2$) and~$\cN(f,z)$, for
  any~$f\in\polsys$ and~$z\in\bP^n$, in~$\cO(N+n^3)$ operations.
\end{lemma}

The cost of the~$k$th iteration in Algorithm~DBP is dominated by the cost of
computing~$\flo{f}_{Q_k}$ and~$\mop{BP}(\left\{ f \right\}_{Q_k})$ and the cost of the
call to~$\mop{HC'}$.  The cost of the call to~$\mop{HC'}$ is dominated by the
cost of the homotopy steps.  Each homotopy step costs~$\cO(N + n^3)$ operations, by
Lemma~\ref{lem:costnewt}.

We now evaluate the cost of computing~$\flo{f}_Q$ and~$\mop{BP}(\left\{ f \right\}_Q)$.
Naturally, the integral part~$A(x)$ of a real number~$x$ is not a rational function
of~$x$ but it can be computed in the BSS model in~$\cO(\log(1+|x|))$
operations using the recursive formula, say for~$x\geq 0$,
\[ A(x) =
  \begin{cases}
    0 & \text{if $x< 1$} \\
    2 A( x / 2 ) & \text{if $x <  2 A( x / 2 )  + 1$}\\
    2 A( x / 2 ) + 1 & \text{otherwise.}
  \end{cases}
\]
Hence we can compute~$Q^{-1} A(Q x)$ in~$\cO(\log Q)$ operations, for any
positive integer~$Q$ and~$x \in [0,1]$.
It follows that one can compute~$\flo{f}_Q$
in~$\cO(N \log Q)$ operations.
The computation of~$\left\{ f \right\}_Q$ is similar
and it is done in~$\cO(N\log Q + N\log N)$ operations,
by Lemma~\ref{lem:trig}, using~$\smash{\tilde S}$ in place of Sibuya's function~$S$.
Finally, given~$\left\{ f \right\}_Q$, one compute~$\BP(\left\{ f \right\}_Q)$ in~$\cO(N^2)$ operations.
In the end, the cost of the~$k$th operation, excluding the call to~$\HCp$, is thus~$\cO( N^2 + N\log Q_k)$ operations.

Hence, the overall cost of the algorithm is
\[ \cO\left( ( N + n^3) K(f)  + \sum_{k=1}^\Omega( N^2 + N\log Q_k) \right) = \cO\left( ( N + n^3) K(f)  + N^2\Omega + N\sum_{k=1}^\Omega\log Q_k \right)  , \]
where~$K(f)$ is the total number of homotopy steps.
By Theorem~\ref{thm:totsteps}, $\bE(K(f)) = \cO(n D^{3/2} N)$, so it only remains to
bound the expectations of~$\Omega$ and~$\sum_{k=1}^\Omega \log Q_k$.

\begin{lemma}
  $\bE(\Omega) \leq 7$
  \label{lem:expW}
\end{lemma}

\begin{proof}
  By Lemma~\ref{lem:sumW}, with~$C=1$, $\delta=1$ and~$L=5$,
  \[ \bE(\Omega) = \sum_{k=1}^\Omega \bP(\Omega \geq k) \leq L+1 + \frac{2^{L+17} N^5 }{N^{2^L} - 1} \leqcalc 7. \qedhere \]
\end{proof}

\begin{lemma}
  $\bE\left( \sum_{k=1}^\Omega \log Q_k \right) \leq 129 \log N$.
  \label{lem:explogpk}
\end{lemma}
\begin{proof}
  Because~$Q_k = N^{2^{k}}$,
  \[ \bE\big( \sum_{k=1}^\Omega \log Q_k \big) = \sum_{k=1}^\infty \log Q_k \, \bP(\Omega \geq k) = \log N \sum_{k=1}^\infty 2^k \, \bP(\Omega \geq k). \]
  Lemma~\ref{lem:sumW}, with~$\delta = 1$, $C=2$ and~$L=5$ gives that
  \[ \sum_{k=1}^\infty 2^k \, \bP(\Omega \geq k) \leq 2^{L+2} + \frac{2^{L+19} N^5 }{N^{2^L} - 2} \leqcalc 129, \]
  where we used that~$N\geq 2$.
\end{proof}

This concludes the proof of Theorem~\ref{thm:main}.

\section*{Conclusion}

The derandomization proposed here relies on extracting randomness from the
input itself, which is made possible by the BSS model and the infinite precision it provides.
Actually, this might also work under finite, and rather moderate, precision.
Indeed, in the~$k$th iteration of Algorithm~\ref{algo:dbp}, we need, very
loosely speaking, about~$\log Q_k$ digits of precision but not much more,
because by construction, the homotopy continuation procedure~$\HCp$ aborts when more precision would be
required for the result to be relevant.  And then, Lemma~\ref{lem:explogpk}
shows that~$\log Q_k$, is typically no more that~$129 \log N$.
So it is
reasonnable to think that, extending the work of
Briquel et al.\autocite{BriCucPen14}, we may run a variant of
Algorithm~\ref{algo:dbp} on a finite precision machine and obtain a significant
probability of success as long as we work with~$C \log N$ digits of precision, for some
constant~$C>0$.

Besides, Armentano et al.\autocite{ArmBelBur16} proposed recently a new
complexity analysis of Beltrán-Pardo algorithm which relies on a refined
homotopy continuation algorithm\autocites{Shu09}{Bel11}.  They obtained a
randomized algorithm that terminates on the average on a random input
after~$\cO(nD^{3/2}N^{1/2})$ homotopy steps.  This is a significant improvement
on the previously known~$\cO(nD^{3/2}N)$ bound.  The derandomization method
 should also apply to this refined algorithm, in all likelihood, but this
is not immediate: to devise the homotopy continuation with precision check, we
had to look deep inside the continuation process. Adapting the method to a
refined homotopy continuation process will inevitably lead to further
difficulties.

\raggedright
\printbibliography

\nopagebreak[4]
\vspace{1cm}
\hfill
\begin{minipage}{.8\linewidth}
  \small
  
   Technische Universität Berlin, 
  Institut für Mathematik,
  Sekretariat MA 3-2\\
  Straße des 17. Juni 136, 10623 Berlin, Deutschland
  
  \medskip
  \emph{E-mail address:} pierre@lairez.fr
  
  \emph{URL:} \href{http://pierre.lairez.fr}{pierre.lairez.fr}

\end{minipage}

\end{document}